\documentclass[12pt,reqno]{amsart}

\usepackage{mathdots}

\usepackage{color}

\usepackage{pdfsync}

\usepackage{enumitem}

\usepackage{wasysym}

\usepackage{amssymb}

\usepackage{mathrsfs}

\DeclareMathAlphabet{\mathpzc}{OT1}{pzc}{m}{it}

\usepackage[all]{xy}

\usepackage{tikz}
\usetikzlibrary{arrows,matrix,decorations.pathmorphing,decorations.pathreplacing,positioning,shapes.geometric,shapes.misc,decorations.markings,decorations.fractals,calc,patterns}

\usepackage{graphicx}

\usepackage{float}

\usepackage[bottom]{footmisc}

\usepackage{moreenum}

\usepackage{makecell}

\entrymodifiers={+!!<0pt,\fontdimen22\textfont2>}

\def\cC{\mathscr{C}}

\def\cT{\mathscr{T}}

\def\cX{\mathscr{X}}

\def\cZ{\mathscr{Z}}

\def\add{\operatorname{add}}

\def\adots{\mathinner{\mkern1mu\raise1.0pt\vbox{\kern7.0pt\hbox{.}}\mkern2mu\raise4.0pt\hbox{.}\mkern2mu\raise7.0pt\hbox{.}\mkern1mu}}

\def\ast{{\textstyle *}}

\def\dddots{\mathinner{\mkern1mu\raise10.0pt\vbox{\kern7.0pt\hbox{.}}\mkern2mu\raise5.3pt\hbox{.}\mkern2mu\raise1.0pt\hbox{.}\mkern1mu}}
\def\dddotssmall{\mathinner{\mkern1mu\raise7.0pt\vbox{\kern7.0pt\hbox{.}}\mkern-1mu\raise4pt\hbox{.}\mkern-1mu\raise1.0pt\hbox{.}\mkern1mu}}

\def\Ext{\operatorname{Ext}}

\def\Hom{\operatorname{Hom}}

\def\inf{\operatorname{inf}}

\def\PSL2{\operatorname{PSL}_2}

\def\SL2{\operatorname{SL}_2}

\def\sup{\operatorname{sup}}

\newcommand{\ZZ}{\mathbb{Z}}
\newcommand{\X}{\mathscr{X}}
\newcommand{\T}{\mathscr{T}}

\newcommand{\CC}{\mathscr{C}}

\newcommand{\Z}{\mathscr{Z}}
\newcommand{\nc}{\operatorname{nc}\nolimits}

\numberwithin{equation}{section}

\newtheorem{Lemma}{Lemma}[section]
\newtheorem{Theorem}[Lemma]{Theorem}
\newtheorem{Proposition}[Lemma]{Proposition}

\theoremstyle{definition}
\newtheorem{Definition}[Lemma]{Definition}
\newtheorem{Setup}[Lemma]{Setup}

\newtheorem{Remark}[Lemma]{Remark}

\newtheorem*{bfhpg*}{}

\newenvironment{VarDescription}[1]%
  {\begin{list}{}{%
    \settowidth{\labelwidth}{\textbf{#1:}}%
    \setlength{\leftmargin}{\labelwidth}\addtolength{\leftmargin}{\labelsep}}}%
  {\end{list}}

\begin{document}

\setlength{\parindent}{0pt}
\setlength{\parskip}{7pt}

\title[Cluster tilting subcategories of Igusa--Todorov cluster categories]{Cluster tilting subcategories and torsion pairs in Igusa--Todorov cluster categories of Dynkin type $A_{ \infty }$} 

\author{Sira Gratz}
\address{Gratz: School of Mathematics and Statistics, University of Glasgow, University Place, Glasgow G12 8SQ, United Kingdom}
\email{Sira.Gratz@glasgow.ac.uk}
\urladdr{https://www.gla.ac.uk/schools/mathematicsstatistics/staff/siragratz}

\author{Thorsten Holm}
\address{Holm: Institut f\"{u}r Algebra, Zahlentheorie und Diskrete
  Mathematik, Fa\-kul\-t\"{a}t f\"{u}r Ma\-the\-ma\-tik und Physik, Leibniz
  Universit\"{a}t Hannover, Welfengarten 1, 30167 Hannover, Germany}
\email{holm@math.uni-hannover.de}
\urladdr{http://www.iazd.uni-hannover.de/\~{ }tholm}

\author{Peter J\o rgensen}
\address{J\o rgensen: School of Mathematics and Statistics,
Newcastle University, Newcastle upon Tyne NE1 7RU, United Kingdom}
\email{peter.jorgensen@ncl.ac.uk}
\urladdr{http://www.staff.ncl.ac.uk/peter.jorgensen}




\begin{abstract}
We give a combinatorial classification of cluster tilting subcategories and torsion pairs in Igusa--Todorov cluster categories of Dynkin type $A_{ \infty }$.
\end{abstract}

\maketitle

\setcounter{section}{-1}
\section{Introduction}
\label{sec:introduction}

Let $\cC( A_n )$ be the cluster category of Dynkin type $A_n$, see \cite[sec.\ 1]{BMRRT} and \cite{CCS}.  It is well known that $\cC( A_n )$ has a combinatorial model by an $( n+3 )$-gon $P$.  The indecomposable objects are in bijection with the diagonals of $P$, and non-vanishing $\Ext^1$ groups correspond to crossing diagonals.  

The combinatorial model has two key properties: Cluster tilting subcategories of $\cC( A_n )$ correspond to triangulations of $P$, and torsion pairs in $\cC( A_n )$ correspond to so-called Ptolemy diagrams in $P$.  The former result is well known and appears to be folklore; the latter is proved in \cite[thm.\ A]{HJR:torsionpairsA}.

The aim of this paper is to prove similar key properties for the cluster categories $\cC( \cZ )$ of Dynkin type $A_{ \infty }$, which were introduced by Igusa and Todorov.  Subsection \ref{subsec:A} is a primer on $\cC( \cZ )$ and its combinatorial model by an $\infty$-gon, and Subsections \ref{subsec:B} and \ref{subsec:C} state the key properties we will prove.

Our results generalise the following parts of the literature:
\begin{itemize}
\setlength\itemsep{4pt}

  \item  When $\cZ$ has one, respectively two limit points (see Definition \ref{D:admissible}(iii)), \cite[thms.\ A,B,C]{HJinfinity}, respectively \cite[thms.\ 3.13, 5.7]{LP} classified cluster tilting subcategories in $\cC( \cZ )$, and showed that they form a cluster structure in the sense of \cite[sec.\ II.1]{BIRS}.

  \item  When $\cZ$ has one, respectively two limit points, \cite[thm.\ 3.18]{Ng}, respectively \cite[thm.\ 4.4]{CZZ} classified torsion pairs in $\cC( \cZ )$.

\end{itemize}
Furthermore, our Theorem \ref{thm:B} is closely related to \cite[thm.\ 7.17]{SvR}.

We would also like to mention that there are a number of papers on the classification of cluster tilting subcategories and torsion pairs in more general cluster categories, mainly based on combinatorial models of Riemann surfaces with marked points on the boundary, see \cite{BZ}, \cite{QZ}, \cite{ZZZ} for surface type and \cite{HJR:tubes} for cluster tubes.

\subsection{The Igusa--Todorov cluster categories $\cC( \cZ )$ of Dynkin type $A_{ \infty }$}
\label{subsec:A}

To explain the categories $\cC( \cZ )$ and their combinatorial models by $\infty$-gons, we first state two definitions.

\begin{Definition}
[Admissible subsets of $S^1$]
\label{D:admissible}
A subset $\cZ$ of the circle $S^1$ is called {\em admissible} if it satisfies the following conditions.
	\begin{enumerate}
	\setlength\itemsep{4pt}
		\item
		$\cZ$ has infinitely many elements.
		\item
		$\cZ \subset S^1$ is a discrete subset, i.e.\ for each $z \in \Z$ there is an open neighbourhood of $z$ in $S^1$, equipped with its usual topology, containing no other element of $\Z$.
		\item
		$\cZ$ satisfies the {\em two-sided limit condition}, i.e.\ each $x \in S^1$ which is the limit of a sequence from $\Z$ is a limit of both an increasing and a decreasing sequence from $\Z$ with respect to the 
		cyclic order.
	\end{enumerate}
Throughout the paper, $\cZ \subset S^1$ is a fixed admissible subset.  We think of $\cZ$ as the vertices of an $\infty$-gon, see Figure \ref{fig:admissible subset}.
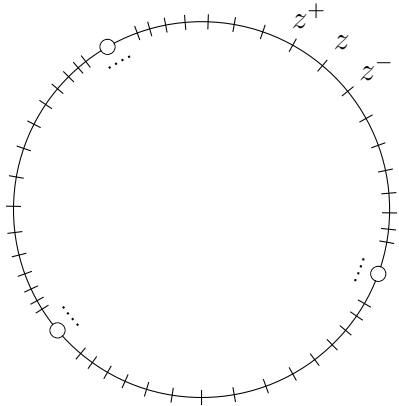
\begin{figure}[H]
\begin{center}
\begin{tikzpicture}[scale=5]
        \draw (0,0) circle(0.5cm);
	
	\node at (50:0.58){$z$};
	\node at (39:0.60){$z^-$};
	\node at (61:0.59){$z^+$};

    \draw  (128:0.48) -- (128:0.52);
    \draw  (131:0.48) -- (131:0.52);
    \draw  (135:0.48) -- (135:0.52);
    \draw  (140:0.48) -- (140:0.52);
    \draw  (146:0.48) -- (146:0.52);
    \draw  (153:0.48) -- (153:0.52);
    \draw  (161:0.48) -- (161:0.52);
	\draw  (170:0.48) -- (170:0.52);
    \draw  (179:0.48) -- (179:0.52);
    \draw  (187:0.48) -- (187:0.52);
    \draw  (194:0.48) -- (194:0.52);
    \draw  (200:0.48) -- (200:0.52);
    \draw  (205:0.48) -- (205:0.52);
    \draw  (209:0.48) -- (209:0.52);
    \draw  (212:0.48) -- (212:0.52);
    
    \draw  (230:0.48) -- (230:0.52);
    \draw  (234.5:0.48) -- (234.5:0.52);
    \draw  (240:0.48) -- (240:0.52);
    \draw  (246:0.48) -- (246:0.52);
    \draw  (253:0.48) -- (253:0.52);
    \draw  (261:0.48) -- (261:0.52);
    \draw  (270:0.48) -- (270:0.52);    
	\draw  (280:0.48) -- (280:0.52);
    \draw  (290:0.48) -- (290:0.52);
    \draw  (299:0.48) -- (299:0.52);
    \draw  (307:0.48) -- (307:0.52);
    \draw  (314:0.48) -- (314:0.52);
    \draw  (320:0.48) -- (320:0.52);
    \draw  (325.5:0.48) -- (325.5:0.52);
    \draw  (330:0.48) -- (330:0.52);

    \draw  (-10:0.48) -- (-10:0.52);
    \draw  (-6:0.48) -- (-6:0.52);
    \draw  (-1:0.48) -- (-1:0.52);
    \draw  (5:0.48) -- (5:0.52);
    \draw  (12:0.48) -- (12:0.52);
    \draw  (20:0.48) -- (20:0.52);
    \draw  (29:0.48) -- (29:0.52);
    \draw  (39:0.48) -- (39:0.52);
	\draw  (50:0.48) -- (50:0.52);
    \draw  (61:0.48) -- (61:0.52);
    \draw  (71:0.48) -- (71:0.52);
    \draw  (80:0.48) -- (80:0.52);
    \draw  (88:0.48) -- (88:0.52);
    \draw  (95:0.48) -- (95:0.52);
    \draw  (101:0.48) -- (101:0.52);
    \draw  (106:0.48) -- (106:0.52);
    \draw  (110:0.48) -- (110:0.52);

	\draw (120:0.5) node[fill=white,circle,inner sep=0.065cm] {} circle (0.02cm);
	\draw (220:0.5) node[fill=white,circle,inner sep=0.065cm] {} circle (0.02cm);
	\draw (340:0.5) node[fill=white,circle,inner sep=0.065cm] {} circle (0.02cm);
	
	\draw[thick,dotted] (115:0.45) arc (115:125:0.45);
	\draw[thick,dotted] (215:0.45) arc (215:225:0.45);
	\draw[thick,dotted] (335:0.45) arc (335:345:0.45);		
	
  \end{tikzpicture}
  \end{center}
  \caption{An example of an admissible subset $\cZ$ of $S^1$, to be thought of as the vertices of an $\infty$-gon, see Definition \ref{D:admissible}.  The points in $\cZ$ converge to the limit points marked with small circles.  Each point $z \in \cZ$ has a predecessor $z^-$ and a successor $z^+$ in $\cZ$, see Remark \ref{rmk:predecessors}. Note that the limit points are not elements of $\Z$ since $\Z$ is discrete.
\label{fig:admissible subset}}
\end{figure}
\end{Definition}
\medskip

\begin{Definition}
[Diagonals]
\label{def:diagonals}
A {\em diagonal of $\cZ$} is a subset $X = \{ x_0,x_1 \} \subset \cZ$ where $x_1 \not\in \{ x_0^-,x_0,x_0^+ \}$.  If $Y = \{ y_0,y_1 \}$ is another diagonal, then $X$ and $Y$ {\em cross} if $x_0 < y_0 < x_1 < y_1$ or $x_0 < y_1 < x_1 < y_0$.  See Definition \ref{def:cyclic} for an explanation of inequalities.

If $D^1$ is the disk bounded by $S^1$, then we think of the diagonal $X$ as an isotopy class of non-selfintersecting curves in $D^1$ between the non-neighbouring vertices $x_0$ and $x_1$, see Figure \ref{fig:cluster tilting subcategory}.  Two diagonals cross if their representing curves intersect in the interior of $D^1$.
\end{Definition}
\medskip

Starting from $\cZ$ and an algebraically closed field $k$, Igusa and Todorov in \cite[sec.\ 2.4]{IT:cyclicposets} constructed a cluster category $\cC( \cZ )$ of Dynkin type $A_{ \infty }$, which has a similar combinatorial model to that of $\cC( A_n )$.  To wit, $\cC( \cZ )$ is a $k$-linear Hom-finite Krull--Schmidt $2$-Calabi--Yau triangulated category; the indecomposable objects are in bijection with the diagonals of $\cZ$, and non-vanishing $\Ext^1$ groups correspond to crossing diagonals.  Further properties of $\cC( \cZ )$ are given in Section \ref{sec:IT}.

\subsection{Cluster tilting subcategories of the cluster categories $\cC( \cZ )$}
\label{subsec:B}

Our first main result is a classification of the cluster tilting subcategories of $\cC( \cZ )$ (see Definition \ref{def:cluster_tilting}).  Cluster tilting subcategories of $\cC( A_n )$  correspond to triangulations of a finite polygon $P$, that is, maximal sets of pairwise non-crossing diagonals of $P$.  By analogy, we expect cluster tilting subcategories of $\cC( \cZ )$ to correspond to triangulations of the $\infty$-gon with vertex set $\cZ$.

This is, in a sense, true, but there is more to say: The definition of admissible subset permits $\cZ$ to have a complicated configuration of limit points, and it is crucial how the endpoints of diagonals converge to the limit points.  Hence the following two definitions.

\begin{Definition}
[The proper limit points of $\cZ$]
We denote by $\overline{\Z}$ the topological closure of $\Z$ in $S^1$, and by 
\[
  L(\Z) = \overline{\Z} \setminus \Z
\]  
the set of proper limit points of $\cZ$.  It is disjoint from $\cZ$ because $\cZ$ is discrete.
\end{Definition}
\medskip

\begin{Definition}
[Leapfrogs and fountains]
\label{def:leapfrog_and_fountain}
Let $\X$ be a set of diagonals of $\Z$.  The following notions are illustrated by Figure \ref{fig:cluster tilting subcategory}.
\begin{figure}
\begin{center}
\begin{tikzpicture}[scale=5,cap=round,>=latex]

		
	\draw[black] (70:0.5) .. controls (70:0.3) and (120:0.3) .. (120:0.5);
	\draw[black] (70:0.5) .. controls (70:0.3) and (115:0.3) .. (115:0.5);
	\draw[black] (70:0.5) .. controls (70:0.3) and (110:0.3) .. (110:0.5);
	\draw[black] (70:0.5) .. controls (70:0.3) and (105:0.3) .. (105:0.5);

	\draw[black] (70:0.5) .. controls (70:0.2) and (230:0.2) .. (230:0.5);
	\draw[black] (70:0.5) .. controls (70:0.2) and (235:0.2) .. (235:0.5);
	\draw[black] (70:0.5) .. controls (70:0.2) and (240:0.2) .. (240:0.5);
	\draw[black] (70:0.5) .. controls (70:0.2) and (245:0.2) .. (245:0.5);

	\draw[thick,dotted] (215:0.4) arc (215:225:0.4);
	
	\draw[black] (70:0.5) .. controls (70:0.2) and (140:0.2) .. (140:0.5);
	\draw[black] (70:0.5) .. controls (70:0.2) and (145:0.2) .. (145:0.5);
	\draw[black] (70:0.5) .. controls (70:0.2) and (150:0.2) .. (150:0.5);
	\draw[black] (70:0.5) .. controls (70:0.2) and (155:0.2) .. (155:0.5);
	\draw[black] (70:0.5) .. controls (70:0.2) and (210:0.2) .. (210:0.5);
	\draw[black] (70:0.5) .. controls (70:0.2) and (205:0.2) .. (205:0.5);
	\draw[black] (70:0.5) .. controls (70:0.2) and (200:0.2) .. (200:0.5);
	\draw[black] (70:0.5) .. controls (70:0.2) and (195:0.2) .. (195:0.5);

	\draw[thick,dotted] (125:0.4) arc (125:135:0.4);
		
	\draw[black] (270:0.5) .. controls (270:0.2) and (45:0.2) .. (45:0.5);
	\draw[black] (270:0.5) .. controls (270:0.2) and (50:0.2) .. (50:0.5);
	\draw[black] (270:0.5) .. controls (270:0.2) and (55:0.2) .. (55:0.5);
	\draw[black] (270:0.5) .. controls (270:0.2) and (60:0.2) .. (60:0.5);

	\draw[thick,dotted] (30:0.4) arc (30:40:0.4);


	\draw[black] (270:0.5) .. controls (270:0.2) and (25:0.2) .. (25:0.5);
	\draw[black] (270:0.5) .. controls (270:0.2) and (20:0.2) .. (20:0.5);
	\draw[black] (270:0.5) .. controls (270:0.2) and (15:0.2) .. (15:0.5);
	\draw[black] (270:0.5) .. controls (270:0.2) and (10:0.2) .. (10:0.5);
	
	\draw[black] (10:0.5) .. controls (0:0.25) and (300:0.25) .. (290:0.5);
	\draw[black] (290:0.5) .. controls (300:0.25) and (-10:0.25) .. (0:0.5);
	\draw[black] (0:0.5) .. controls (-10:0.26) and (310:0.26) .. (300:0.5);
	\draw[black] (300:0.5) .. controls (310:0.28) and (-20:0.28) .. (-10:0.5);
	\draw[black] (-10:0.5) .. controls (-20:0.3) and (320:0.3) .. (310:0.5);
	
	\draw[thick,dotted] (330:0.45) -- (330:0.35);

        \draw[black] (0,0) circle(0.5cm);
        
        \node at (220:0.65) {$a_1$};
	\draw (220:0.5) node[fill=white,circle,inner sep=0.065cm] {} circle (0.02cm);	        
        \node at (130:0.65) {$a_2$};
	\draw (130:0.5) node[fill=white,circle,inner sep=0.065cm] {} circle (0.02cm);
	\node at (35:0.65) {$a_3$};
	\draw (35:0.5) node[fill=white,circle,inner sep=0.065cm] {} circle (0.02cm);	
	\node at (330:0.65) {$a_4$};
	\draw (330:0.5) node[fill=white,circle,inner sep=0.065cm] {} circle (0.02cm);	
  \end{tikzpicture}
\caption{A set of diagonals of $\Z$ with fountains converging to the limit points $a_1, a_2$, $a_3$ and a leapfrog converging to the limit point $a_4$, see Definition \ref{def:leapfrog_and_fountain}.  Such convergence must occur in each cluster tilting subcategory of $\CC(\Z)$ by Theorem \ref{thm:B}.}
\label{fig:cluster tilting subcategory}
\end{center}
\end{figure}
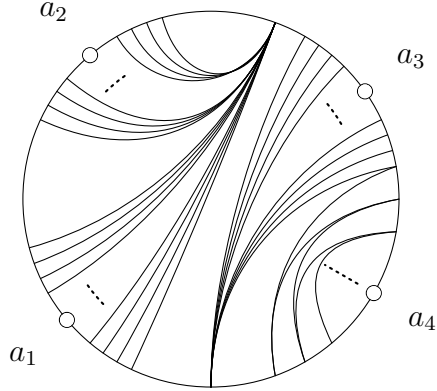
\begin{itemize}
\setlength\itemsep{4pt}

  \item  Given $a \in L( \cZ )$, we say that $\X$ {\em has a leapfrog converging to $a\in L(\Z)$} if there is a sequence $\{x_i,y_i\}_{i \in \ZZ_{\geqslant 0}}$ of diagonals from $\X$ with $x_i \to a$ from below and $y_i \to a$ from above.  (Convergence from below and above is explained in Definition \ref{def:convergence_from_below_and_above}.)

  \item  Given $a \in L(\Z)$, $z \in \Z$.  We say that $\X$ {\em has a right fountain at $z$ converging to $a$} if there is a sequence $\{z, x_i\}_{i \in \ZZ_{\geqslant 0}}$ from $\X$ with $x_i \to a$ from below. We say that $\X$ {\em has a left fountain at $z$ converging to $a$} if there is a sequence $\{z, y_i\}_{i \in \ZZ_{\geqslant 0}}$ from $\X$ with $y_i \to a$ from above.
\medskip
	
\noindent
We say that $\X$ {\em has a fountain at $z$ converging to $a$} if it has a right fountain and a left fountain at $z$ converging to $a$.

\end{itemize}
\end{Definition}
\medskip

Here is our first main result.  It is closely related to \cite[thm.\ 7.17]{SvR}.  Given a set $\cX$ of diagonals of $\cZ$, we write $E(\X)$ for the corresponding set of indecomposable objects of $\CC(\Z)$.

\begin{Theorem}
[=Theorem \ref{T:cluster tilting}]
\label{thm:B}
Let $\X$ be a set of diagonals of $\Z$. Then $\add E(\X)$ is a cluster tilting sub\-ca\-te\-go\-ry if and only if $\X$ is a maximal set of pairwise non-crossing diagonals, such that for each $a \in L(\Z)$, the set $\X$ has a fountain or a leapfrog converging to $a$.
\end{Theorem}

One of the salient features of cluster tilting subcategories are their nice combinatorial properties encoded in the notion of {\em cluster structure}.  We thank Adam-Christiaan van Roosmalen for pointing out that the following result  follows from \cite[thm.\ 5.6]{SvR}.  We will give a direct proof.

\begin{Theorem}
[=Theorem \ref{thm:cluster_structure}]
\label{thm:C}
The cluster tilting subcategories of $\cC( \cZ )$ form a cluster structure in the sense of \cite[sec.\ II.1]{BIRS}.
\end{Theorem}

To get this from \cite[thm.\ 5.6]{SvR} requires the existence of a so-called directed cluster tilting subcategory of $\cC( \cZ )$, which can be obtained from Theorem \ref{thm:B} by picking a vertex $z \in \cZ$ and letting $\cX$ be the set of all diagonals from $z$ to non-neighbouring vertices.

\subsection{Torsion pairs in the cluster categories $\cC( \cZ )$}
\label{subsec:C}

Our second main result is a classification of the torsion pairs in $\cC( \cZ )$ (see Definition \ref{def:torsion_pairs}).  Recall that torsion pairs in $\cC( A_n )$ correspond to so-called Ptolemy diagrams in a finite polygon $P$,
see \cite[thm.\ A]{HJR:torsionpairsA}.  Again there is an analogue for $\cC( \cZ )$, and again, convergence plays a crucial role.  Hence the following definition.

\begin{Definition}
[Conditions PC1 and PC2]
\label{def:PC}
We can impose the following conditions on a set $\X$ of diagonals of $\Z$, see Figure \ref{fig:PC}.
	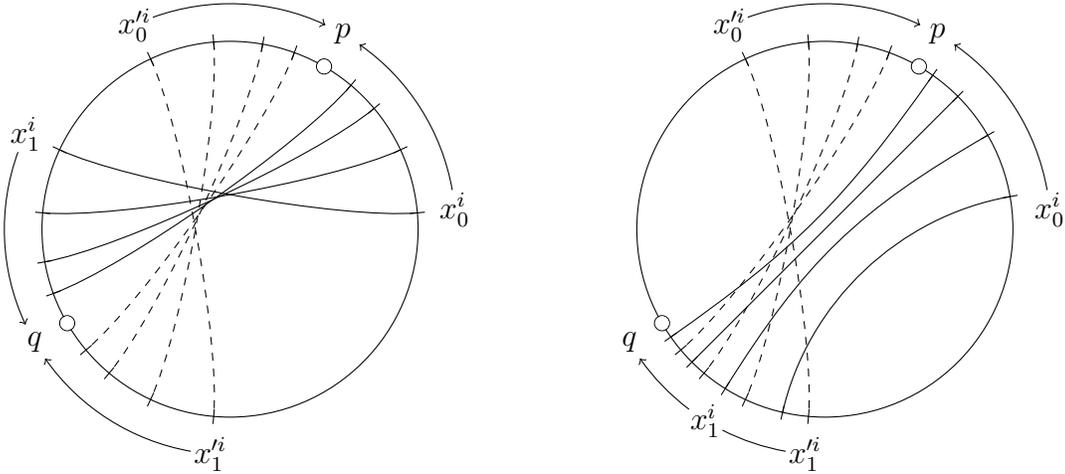
\begin{figure}
\begin{center}
\begin{tikzpicture}[scale=5]

    \draw (0,0) circle(0.5cm);
	
	\node at (5:0.60){$x_0^i$};
	\node at (60:0.60){$p$};
	\node at (115:0.60){$x'^i_0$};

	\node at (155:0.60){$x_1^i$};
	\node at (210:0.60){$q$};
	\node at (265:0.60){$x'^i_1$};

	\draw[->] (10:0.60) arc (10:55:0.60);
	\draw[<-] (65:0.60) arc (65:110:0.60);

	\draw[->] (160:0.60) arc (160:205:0.60);
	\draw[<-] (215:0.60) arc (215:260:0.60);

	\draw (210:0.5) node[fill=white,circle,inner sep=0.065cm] {} circle (0.02cm);
	\draw (60:0.5) node[fill=white,circle,inner sep=0.065cm] {} circle (0.02cm);

    \draw (5:0.48) -- (5:0.52);
    \draw (155:0.48) -- (155:0.52);    
    \draw (5:0.5) .. controls (5:0.3) and (155:0.3) .. (155:0.5);	    

    \draw (25:0.48) -- (25:0.52);
    \draw (175:0.48) -- (175:0.52);    
    \draw (25:0.5) .. controls (25:0.3) and (175:0.3) .. (175:0.5);	    

    \draw (40:0.48) -- (40:0.52);
    \draw (190:0.48) -- (190:0.52);    
    \draw (40:0.5) .. controls (40:0.3) and (190:0.3) .. (190:0.5);	    

    \draw (50:0.48) -- (50:0.52);
    \draw (200:0.48) -- (200:0.52);    
    \draw (50:0.5) .. controls (50:0.3) and (200:0.3) .. (200:0.5);

    \draw (115:0.48) -- (115:0.52);
    \draw (265:0.48) -- (265:0.52);    
    \draw[dashed] (115:0.5) .. controls (115:0.3) and (265:0.3) .. (265:0.5);
    
    \draw (95:0.48) -- (95:0.52);
    \draw (245:0.48) -- (245:0.52);    
    \draw[dashed] (95:0.5) .. controls (95:0.3) and (245:0.3) .. (245:0.5);	    

    \draw (80:0.48) -- (80:0.52);
    \draw (230:0.48) -- (230:0.52);    
    \draw[dashed] (80:0.5) .. controls (80:0.3) and (230:0.3) .. (230:0.5);	    
    
    \draw (70:0.48) -- (70:0.52);
    \draw (220:0.48) -- (220:0.52);    
    \draw[dashed] (70:0.5) .. controls (70:0.3) and (220:0.3) .. (220:0.5);

\begin{scope}[xshift = 45]

    \draw (0,0) circle(0.5cm);
	
	\node at (5:0.60){$x_0^i$};
	\node at (60:0.60){$p$};
	\node at (115:0.60){$x'^i_0$};

	\node at (238:0.60){$x_1^i$};
	\node at (210:0.60){$q$};
	\node at (265:0.60){$x'^i_1$};

	\draw[->] (10:0.60) arc (10:55:0.60);
	\draw[<-] (65:0.60) arc (65:110:0.60);

	\draw (260:0.60) arc (260:243:0.60);
	\draw[->] (233:0.60) arc (233:215:0.60);

	\draw (210:0.5) node[fill=white,circle,inner sep=0.065cm] {} circle (0.02cm);
	\draw (60:0.5) node[fill=white,circle,inner sep=0.065cm] {} circle (0.02cm);

    \draw (10:0.48) -- (10:0.52);
    \draw (257:0.48) -- (257:0.52);    
    \draw (10:0.5) .. controls (10:0.2) and (257:0.2) .. (257:0.5);	    

    \draw (30:0.48) -- (30:0.52);
    \draw (238:0.48) -- (238:0.52);    
    \draw (30:0.5) .. controls (30:0.1) and (238:0.1) .. (238:0.5);	    

    \draw (45:0.48) -- (45:0.52);
    \draw (225:0.48) -- (225:0.52);    
    \draw (45:0.5) .. controls (45:0.1) and (225:0.1) .. (225:0.5);	    

    \draw (55:0.48) -- (55:0.52);
    \draw (215:0.48) -- (215:0.52);    
    \draw (55:0.5) .. controls (55:0.1) and (215:0.1) .. (215:0.5);

    \draw (115:0.48) -- (115:0.52);
    \draw (265:0.48) -- (265:0.52);    
    \draw[dashed] (115:0.5) .. controls (115:0.3) and (265:0.3) .. (265:0.5);
    
    \draw (95:0.48) -- (95:0.52);
    \draw (245:0.48) -- (245:0.52);    
    \draw[dashed] (95:0.5) .. controls (95:0.3) and (245:0.3) .. (245:0.5);	    

    \draw (80:0.48) -- (80:0.52);
    \draw (230:0.48) -- (230:0.52);    
    \draw[dashed] (80:0.5) .. controls (80:0.3) and (230:0.3) .. (230:0.5);	    
    
    \draw (70:0.48) -- (70:0.52);
    \draw (220:0.48) -- (220:0.52);    
    \draw[dashed] (70:0.5) .. controls (70:0.3) and (220:0.3) .. (220:0.5);	    

\end{scope}

  \end{tikzpicture}
\end{center}
\caption{Illustration of conditions PC1 and PC2 from Definition \ref{def:PC}.}
\label{fig:PC}
\end{figure}
The letters ``PC'' stands for ``precovering''.
\begin{VarDescription}{PC2:\quad}
\setlength\itemsep{4pt}

  \item[PC1:]  If there is a sequence $\{x^i_0, x^i_1\}_{i \in \ZZ_{\geqslant 0}}$ from $\X$ with $x^i_0 \to p$ from below and $x^i_1 \to q$ from below with $p \neq q$, then there is a sequence $\{x'^i_0, x'^i_1\}_{i \in \ZZ_{\geqslant 0}}$ from $\X$ with $x'^i_0 \to p$ from above and $x'^i_1 \to q$ from above.
  
  \item[PC2:]If there is a sequence $\{x^i_0, x^i_1\}_{i \in \ZZ_{\geqslant 0}}$ from $\X$ with $x^i_0 \to p$ from below and $x^i_1 \to q$ from above with $p \neq q$, then there is a sequence $\{x'^i_0, x'^i_1\}_{i \in \ZZ_{\geqslant 0}}$ from $\X$ with $x'^i_0 \to p$ from above and $x'^i_1 \to q$ from above.
  
\end{VarDescription}

\end{Definition}
\medskip

The following combinatorial notion was introduced in \cite[def.\ 0.3]{Ng}.

\begin{Definition}
[The Ptolemy condition]
\label{def:Ptolemy}
 Let $\X$ be a set of diagonals of $\Z$.
We say that {\em $\X$ satisfies the Ptolemy condition} if, whenever $\{x_0,x_1\} \in \X$ and $\{y_0,y_1\} \in \X$ cross, then those of $\{x_0,y_0\}$, $\{x_0,y_1\}$, $\{x_1,y_0\}$ and $\{x_1,y_1\}$ which are diagonals of $\Z$ (i.e.\ whose vertices are non-neighbouring) also lie in $\X$.  See Figure \ref{fig:Ptolemy}.
\begin{figure}
\begin{center}
\begin{tikzpicture}[scale=5,cap=round,>=latex]

        \draw[black] (0,0) circle(0.5cm);

  \draw (50:0.48) -- (50:0.52);
  \node at (50:0.6) {$x_0$};
  \draw (110:0.48) -- (110:0.52);
  \node at (110:0.6) {$y_0$};  
  \draw (210:0.48) -- (210:0.52);
  \node at (210:0.6) {$x_1$};  
  \draw (300:0.48) -- (300:0.52);
  \node at (300:0.6) {$y_1$};  

  \draw[black] (50:0.5) .. controls (50:0.2) and (210:0.2) .. (210:0.5);
  \draw[black] (110:0.5) .. controls (110:0.2) and (300:0.2) .. (300:0.5);

  \draw[dashed,black] (50:0.5) .. controls (50:0.25) and (110:0.25) .. (110:0.5);
  \draw[dashed,black] (50:0.5) .. controls (50:0.2) and (300:0.2) .. (300:0.5);
  \draw[dashed,black] (210:0.5) .. controls (210:0.2) and (110:0.2) .. (110:0.5);
  \draw[dashed,black] (210:0.5) .. controls (210:0.2) and (300:0.2) .. (300:0.5);

  \end{tikzpicture}
\caption{The Ptolemy condition from Definition \ref{def:Ptolemy}: If the crossing diagonals $\{ x_0,x_1 \}$ and $\{ y_0,y_1 \}$ are in $\cX$, then so are those of $\{ x_0,y_0 \}$, $\{ x_0,y_1 \}$, $\{ x_1,y_0 \}$, $\{ x_1,y_1 \}$ which are diagonals.}
\label{fig:Ptolemy}
\end{center}
\end{figure}
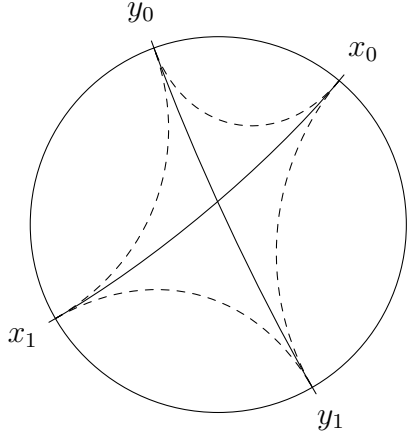
\end{Definition}
\medskip

Here is our second main result.

\begin{Theorem}
[=Theorem \ref{T:torsion pair}]
\label{thm:A}
Let $\X$ be a set of diagonals of $\cZ$. Then $\add E(\X)$ is the first half of a torsion pair in $\CC(\Z)$ if and only if $\X$ satisfies conditions PC1, PC2, and the Ptolemy condition.
\end{Theorem}

Note that the first half of a torsion pair determines the second half, so our result does provide a complete classification.

\begin{Remark}
\label{rmk:PC}
The conjunction of PC1 and PC2 is equivalent to the following condition.
\begin{VarDescription}{PC:\quad}
\setlength\itemsep{4pt}

  \item[PC:]  If there is a sequence $\{x^i_0, x^i_1\}_{i \in \ZZ_{\geqslant 0}}$ from $\X$ with $x^i_0 \to p$ from below and $x^i_1 \to q$ with $p \neq q$, then there is a sequence $\{x'^i_0, x'^i_1\}_{i \in \ZZ_{\geqslant 0}}$ from $\X$ with $x'^i_0 \to p$ from above and $x'^i_1 \to q$ from above.
    
\end{VarDescription}
It is clear that PC implies PC1 and PC2.  To see the converse, note that the sequence $\{x^i_0, x^i_1\}_{i \in \ZZ_{\geqslant 0}}$ in PC will either have a subsequence with $x_1^i \rightarrow q$ from below, and then PC1 can be applied, or a subsequence with $x_1^i \rightarrow q$ from above, and then PC2 can be applied. 
\end{Remark}

The paper is organised as follows: Section \ref{sec:admissible_sets} shows some properties of admissible subsets of $S^1$.  Section \ref{sec:IT} recalls the cluster category $\cC( \cZ )$ from \cite[sec.\ 2.4]{IT:cyclicposets}.  Section \ref{sec:precovering} provides a main ingredient for the proof of Theorem \ref{thm:A} by showing that $\add E( \cX )$ is precovering if and only if $\cX$ satisfies conditions PC1 and PC2.  Section \ref{sec:torsion_pairs} proves Theorem \ref{thm:A}.  Section \ref{sec:cluster_tilting} proves Theorems \ref{thm:B} and \ref{thm:C}.

\section{Admissible subsets of the circle $S^1$}
\label{sec:admissible_sets}

\begin{Definition}
[Cyclically ordered subsets of $S^1$]
\label{def:cyclic}
The circle $S^1$, equipped with its usual topology and orientation, has a natural structure as a cyclically ordered set.

We choose anticlockwise as the positive direction, whence the inequalities
$x_0 < x_1 < \ldots < x_n$ mean that, when moving anticlockwise around the circle, after encountering $x_{i-1}$ for $i = 1, \ldots, n$, the next element of $\{x_0, \ldots, x_n\}$ encountered is precisely $x_i$.  See Figure \ref{fig:order}.  Soft inequalities are defined analogously.

The cyclic order permits to define closed or (half) open intervals of $S^1$; for instance, the closed interval $[ a,b ]$ is shown in Figure \ref{fig:order}.  Each interval has an induced linear order.

The cyclic order on $S^1$ induces a cyclic order on each subset of $S^1$, in particular on $\cZ$.  
\begin{figure}
\begin{center}
\begin{tikzpicture}[scale=5.0,cap=round,>=latex]

	\node at (60:0.58cm) {$x_0$};
	\draw (60:0.48) -- (60:0.52);
	\node at (80:0.58cm) {$x_1$};
	\draw (80:0.48) -- (80:0.52);
	\node at (103:0.58cm) {$x_2$};
	\draw (103:0.48) -- (103:0.52);
	\node at (185:0.58cm) {$x_3$};
	\draw (185:0.48) -- (185:0.52);
	\node at (262:0.58cm) {$x_4$};
	\draw (262:0.48) -- (262:0.52);
	\node at (32:0.58cm) {$x_5$};
	\draw (32:0.48) -- (32:0.52);
        \draw (0,0) circle(0.5cm);
  
     \begin{scope}[xshift = 40]   

	\node at (60:0.59cm) {$a$};
	\node at (180:0.59cm) {$b$};
        \draw (0,0) circle(0.5cm);
        \draw[ultra thick] (60:0.48) -- (60:0.52);
        \draw[ultra thick] (180:0.48) -- (180:0.52);
        \draw[ultra thick] (60:0.5) arc (60:180:0.5);

      \end{scope}
  \end{tikzpicture}
\end{center}
\caption{Illustration of Definition \ref{def:cyclic}.  The elements $x_0, \ldots, x_5$ of $S^1$ satisfy  $x_0 < x_1 < x_2 < x_3 < x_4 < x_5$, and the interval $[a,b]$ is marked by a thick arc.}
\label{fig:order}
\end{figure}
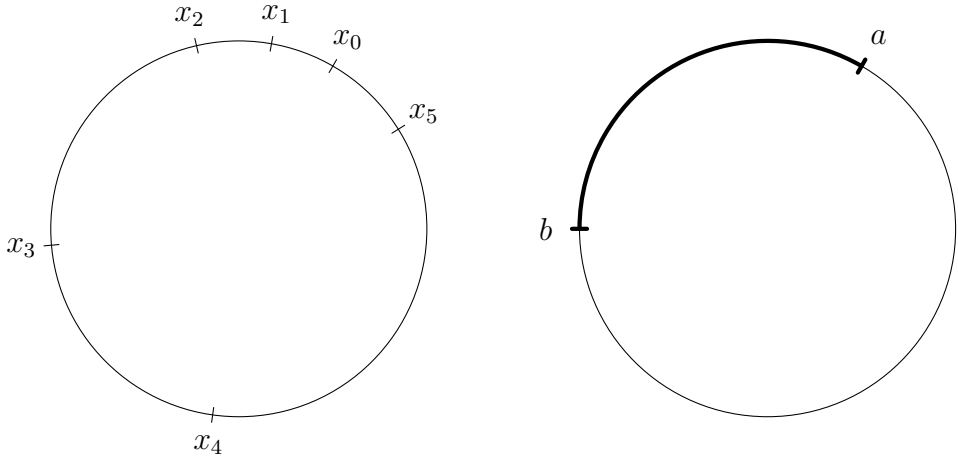
\end{Definition}
\medskip

\begin{Remark}
[Predecessors and successors in $\cZ$]
\label{rmk:predecessors}
It follows directly from Definition \ref{D:admissible} that:
\begin{itemize}
\setlength\itemsep{4pt}

  \item  Each $z \in \Z$ has a unique predecessor $z^- \in \Z$, i.e.\ a unique element $z^- \in \Z$ such that $(z^-,z) \cap \Z = \varnothing$.

  \item  Each $z \in \Z$ has a unique successor $z^+ \in \Z$, i.e.\ a unique element $z^+ \in \Z$ such that $(z,z^+) \cap \Z = \varnothing$. 

\end{itemize}
Figure \ref{fig:admissible subset} in the introduction shows an example of an admissible subset and of the predecessor and successor of one of its elements.
\end{Remark}
\medskip

\begin{Remark}
[A dichotomy for sequences in $\cZ$]
\label{R:dichotomy}
Since $\Z$ is discrete, each sequence $\{z_i\}_{i \in \ZZ_{\geqslant 0}}$ from $\Z$ which converges to a $z \in \Z$ has to satisfy $z_i = z$ for $i \gg 0$. Thus each convergent sequence from $\Z$ that is not constant from some step converges to an element of $L(\Z)$. Furthermore, since $S^1$ is compact, each sequence $\{z_i\}_{i \in \ZZ_{\geqslant 0}}$ from $\Z$ has a convergent subsequence $\{z'_i\}_{i \in \ZZ_{\geqslant 0}}$ converging to some point in $\overline{\Z}$.

There is hence a dichotomy:
\begin{itemize}
\setlength\itemsep{4pt}

  \item Either the subsequence $\{z'_i\}_{i \in \ZZ_{\geqslant 0}}$ converges to $z \in \Z$, and $z'_i$ is constant from some step,

  \item or the subsequence $\{z'_i\}_{i \in \ZZ_{\geqslant 0}}$ 	converges to a proper limit point $a \in L(\Z)$ and $z'_i$ is not constant from any step.
  
\end{itemize}
In the latter case, by refining the sequence further if necessary, we can suppose that the sequence is increasing 	(i.e.\ $z'_0 \leqslant z'_1 \leqslant \ldots \leqslant z'_k < a$ for each $k \in \ZZ_{\geqslant 0}$) or decreasing (i.e.\ $z'_0 \geqslant z'_1 \geqslant \ldots \geqslant z'_k > a$ for each $k \in \ZZ_{\geqslant 0}$). 
\end{Remark}
\medskip

\begin{Definition}
[Convergence from below and above]
\label{def:convergence_from_below_and_above}
Let $\{z_i\}_{i \in \ZZ_{\geqslant 0}}$ be a convergent sequence from $\Z$. If $\{z_i\}_{i \in \ZZ_{\geqslant 0}}$ converges to $p \in \overline{\Z}$, then we write $z_i \to p$.  
\begin{itemize}
\setlength\itemsep{4pt}

  \item  We say that $z_i \to p$ {\em from below} if there is a $\mu \in S^1 \setminus \{p\}$ such that $z_i \in [\mu, p]$ from some step.

  \item  We say that $z_i \to p$ {\em from above} if there is a $\nu \in S^1 \setminus \{p\}$ such that $z_i \in [p, \nu]$ from some step.

\end{itemize}
If $z_i \to p$ with $p \in \Z$, then $z_i = p$ from some step by Remark \ref{R:dichotomy}, so $z_i \to p$ from below {\em and} from above. 
\end{Definition}
\medskip

\begin{Definition}
[Infimum and supremum]
Let $a,b\in S^1$.  Each non-empty subset $P \subseteq [a,b] \cap \Z \subset S^1$ has an infimum and a supremum, and there is a decreasing sequence in $P$
converging to its infimum, denoted by $\inf_{[a,b]} P$, and an increasing sequence in $P$ converging to its supremum, denoted by $\sup_{[a,b]} P$. 

Note that the infimum and the supremum are contained in the interval $[a,b]$, but not necessarily in $P$ or in $\Z$.  If $i = \inf_{[a,b]} P$  and $s=\sup_{[a,b]} P$ then $a \leqslant i \leqslant p \leqslant s \leqslant b$ for each $p \in P$.

Note that any increasing or decreasing sequence in an interval $[a,b]$ is convergent to a point in that interval.
\end{Definition}
\medskip

Recall that $\Z$ contains infinitely many points by Definition \ref{D:admissible}(i).  For each $z \in \Z$, the sequence $\{ z^{+n} \}_{ n \geqslant 0 }$ defined iteratively by $z^{ +0 } = z$ and $z^{ +(k+1) } = ( z^{ +k } )^+$ for each $k \in \ZZ_{\geqslant 0}$ is an increasing sequence.  Moreover, there are infinitely many points of $\Z$ in $[z,z^-]$ whence $z \leqslant z^{ +n } < z^-$.  So $\{z^{+n}\}_{n \geqslant 0}$ is an increasing sequence in $[z,z^-]$ and it must converge to a limit point.

\begin{Definition}
The limit point of $\{z^{+n}\}_{n \geqslant 0}$ will be denoted $z^{ +\infty }$.  Symmetrically, we can define $\{ z^{ -n } \}_{n \geqslant 0}$ and its limit point will be denoted $z^{ -\infty }$.
\end{Definition}
\medskip

\begin{Lemma}\label{L:limit}
We have $[z,z^{+\infty}] \cap L(\Z) = \{z^{+\infty}\}$ and $[z^{-\infty},z] \cap L(\Z) = \{z^{-\infty}\}$.
\end{Lemma}

\begin{proof}
	We only prove that $[z,z^{+\infty}] \cap L(\Z) = \{z^{+\infty}\}$; the equality $[z^{-\infty},z] \cap L(\Z) = \{z^{-\infty}\}$ is proved symmetrically. 
	The inclusion $\{z^{+\infty}\} \subseteq [z,z^{+\infty}] \cap L(\Z)$ is clear by definition. The inclusion $[z,z^{+\infty}] \cap L(\Z) \subseteq \{z^{+\infty}\}$ amounts to showing that 
	$[z,z^{+\infty}) \cap L(\Z) = \varnothing$, which again amounts to showing that $(z,z^{+\infty}) \cap L(\Z) = \varnothing$, since $z\in \Z$, and hence $z \notin L(\Z)$. 
	So suppose for a contradiction that there exists $x\in (z,z^{+\infty}) \cap L(\Z)$. In particular,
	$x\notin\Z$ and there exists a sequence 
	$\{z_i\}_{i \in \ZZ_{\geqslant 0}}$ from $\Z$ converging to $x$. 
	By construction we have that $z^{+\infty} = \sup_{[z,z^-]}\{z^{+n} \mid n \geqslant 0\}$, so 
	we can find $m \geqslant 0$ such that $z^{+m} < x < z^{+(m+1)}$ (note that $x$ can not equal any of the $z^{+n}$
	since $x\notin \Z$).
But since the sequence $\{z_i\}_{i \in \ZZ_{\geqslant 0}}$ converges to $x$, the open neighbourhood
$(z^{+m},z^{+(m+1)})$ of $x$ contains infinitely many entries of the sequence $\{z_i\}_{i \in \ZZ_{\geqslant 0}}$.
Since the $z_i$ are in $\Z$, this clearly contradicts the definition of $z^{+(m+1)}$.
\end{proof}

\section{The Igusa--Todorov cluster categories $\CC(\Z)$ of Dynkin type $A_{ \infty }$}
\label{sec:IT}

\begin{Setup}
In the rest of the paper, $k$ is an algebraically closed field.
\end{Setup}
\medskip

Igusa and Todorov \cite{IT:cyclicposets} constructed a cluster category $\cC( \cZ )$.  They proved in \cite[sec.\ 2.4]{IT:cyclicposets} that it has the following properties.
\begin{enumerate}
\setlength\itemsep{4pt}

  \item  $\cC( \cZ )$ is a $k$-linear Hom-finite Krull--Schmidt triangulated category.

  \item  $\cC( \cZ )$ is $2$-Calabi--Yau, that is, there are natural isomorphisms
\[  
  \Ext_{ \cC( \cZ ) }^1(X,Y) \cong \mathrm{D}\Ext_{ \cC( \cZ ) }^1(Y,X)
\]
where $\mathrm{D}(-) = \Hom_k( -,k )$.
  
  \item  If $X = \{ x_0,x_1 \}$ is a diagonal of $\cZ$, then there is an indecomposable object $E( X ) = E( x_0,x_1 )$ in $\cC( \cZ )$, and this induces a bijection from diagonals of $\cZ$ to isomorphism classes of indecomposable objects of $\cC( \cZ )$.  

  \item  The suspension functor acts on the indecomposable objects $E(X)$ by
	\[
		\Sigma( E( x_0,x_1 ) ) = E( x_0^-,x_1^- ).
	\]

  \item  We have
\begin{equation*}
\label{eqn:Ext-spaces}
  \Ext_{ \CC(\Z) }^1( E(X),E(Y) ) \cong
	\begin{cases}
		k \text{ if $X$ and $Y$ cross,}\\
		0 \text{ otherwise. }
	\end{cases}
\end{equation*}

  \item  Since
$\Hom_{\CC(\Z)}(E(X),E(Y)) \cong \Ext_{\CC(\Z)}^1(E(X),\Sigma^{-1}E(Y))$,
it follows from (iv) and (v) that $\Hom_{ \CC(\Z) }( E(X),E(Y) )$ is isomorphic to
\[
	\begin{cases}
		k \text{ if we can write $X = \{ x_0,x_1 \}$ and $Y = \{ y_0,y_1 \}$ with $x_0 \leqslant y_0 \leqslant x_1^{ -- } < x_1 \leqslant y_1 \leqslant x_0^{--}$, }\\
		0 \text{ otherwise. }
	\end{cases}
\]

  \item  In part (vi), if $X = \{ x_0,x_1 \}$ and $Y = \{ y_0,y_1 \}$ with $x_0 \leqslant y_0 \leqslant x_1^{ -- } < x_1 \leqslant y_1 \leqslant x_0^{--}$, then a morphism $E(X) \to E(Y)$ factors through $E(S)$ if and only if we can write $S = \{ s_0,s_1 \}$ with $x_0 \leqslant s_0 \leqslant y_0$ and $x_1 \leqslant s_1 \leqslant y_1$.

\end{enumerate}

In part (v) observe that non-vanishing of $\Ext^1$ is symmetric in the two arguments, as indeed it must be by the $2$-Calabi--Yau property from (ii).
Figure \ref{fig:morphisms} provides an illustration of morphisms between indecomposable objects. 

\begin{figure}
\begin{center}
\begin{tikzpicture}[scale=5,cap=round,>=latex]
        \draw (0,0) circle(0.5cm);
	\node at (60:0.58) {$x_0$};
	\node at (180:0.58) {$y_1$};
	\node at (250:0.58) {$x_1$};
	\node at (-20:0.58) {$y_0$};
	\node at (150:0.58) {$z_1$};
        \draw (60:0.48) -- (60:0.52);
        \draw (70:0.48) -- (70:0.52);
        \draw (50:0.48) -- (50:0.52);
        \draw (180:0.48) -- (180:0.52);
         \draw (190:0.48) -- (190:0.52);
          \draw (170:0.48) -- (170:0.52);
        \draw (240:0.48) -- (240:0.52);
         \draw (260:0.48) -- (260:0.52);
          \draw (250:0.48) -- (250:0.52);
        \draw (-20:0.48) -- (-20:0.52);
        \draw (-10:0.48) -- (-10:0.52);
	\draw (-30:0.48) -- (-30:0.52);
        \draw (140:0.48) -- (140:0.52);
        \draw (150:0.48) -- (150:0.52);
        \draw (160:0.48) -- (160:0.52);

    \draw (60:0.5) .. controls (60:0.3) and (250:0.3) .. (250:0.5);	
    \draw (180:0.5) .. controls (180:0.3) and (-20:0.3) .. (-20:0.5);
    \draw (60:0.5) .. controls (60:0.3) and (150:0.3) .. (150:0.5);
	
  \end{tikzpicture}
\end{center}
\caption{The non-zero morphism spaces between the indecomposable objects corresponding to the pictured diagonals are precisely $\Hom(E(x_0,x_1),E(y_0,y_1))$, $\Hom(E(y_0,y_1),E(x_0,x_1))$ and $\Hom(E(x_0,z_1),E(x_0,x_1))$, as well as the endomorphism spaces of each of the three indecomposable objects. All other morphism spaces between these three objects are zero.  See Section \ref{sec:IT}(vi).}
\label{fig:morphisms}
\end{figure}
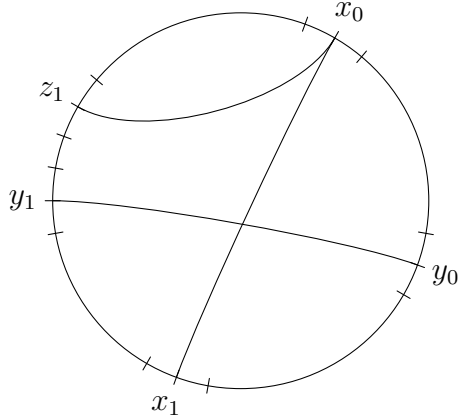

\section{Precovering subcategories of the cluster categories $\CC(\Z)$} 
\label{sec:precovering}

This section provides the following main ingredient for the proof of Theorem \ref{thm:A}.

\begin{Theorem}
\label{T:precovering}
Let $\X$ be a set of diagonals of $\Z$. Then $\add E(\X)$ is a precovering subcategory of $\CC(\Z)$ if and only if $\X$ satisfies conditions PC1 and PC2 from Definition \ref{def:PC}.
\end{Theorem}

The proof can be found at the end of the section.  First we require some preparation, not least the following definition due to \cite[sec.\ 1]{E}.

\begin{Definition}
[Precovers]
\label{D:precovering}
	Let $\T$ be a category, $X\subseteq \T$ a full subcategory.  
	\begin{enumerate}
	\setlength\itemsep{4pt}
	\item Let $t\in \T$ be an object. An object $x\in X$ together with a morphism $f \colon x \to t$
	is called an {\em $X$-precover of $t$} if each morphism 
	$g \colon x' \to t$ with $x' \in X$ factors through $f$.  That is, there exists a morphism $h \colon x' \to x$ 
	such that $g = f\circ h$.
	\[
	 \xymatrix{& x \ar[d]^f \\
		  x' \ar[r]_-g \ar@{-->}[ru]^h & t}
	\]	
\item The subcategory $X \subseteq \T$ is called {\em precovering} if each object 
$t\in \T$ has an $X$-precover. 	
\end{enumerate}
\end{Definition}
\medskip

\begin{Definition}
Let $\cT$ be an additive category.  An {\em additive subcategory} $X$ of $\cT$ is a full subcategory of $\cT$ closed under isomorphisms, finite direct sums, and direct summands.
\end{Definition}
\medskip

\begin{Remark}
Since $\CC(\Z)$ is Krull-Schmidt, its additive subcategories are determined by the indecomposable objects they contain. Thus, there is a one-to-one correspondence between additive subcategories of $\CC(\Z)$ and sets of diagonals of $\Z$.

Given a set of diagonals $\X$ we write $E(\X)$ for the corresponding set of indecomposable objects of $\CC(\Z)$.  The corresponding additive subcategory of $\CC(\Z)$ is given by $\add E(\X)$. 
\end{Remark}
\medskip

\begin{Lemma}
\label{L:precover decomposition}
Let $D \subseteq \CC(\Z)$ be an additive subcategory, $e \in \CC(\Z)$ an indecomposable object, and 
\[
  \delta \colon d_1 \oplus \ldots \oplus d_n \to e
\]
a morphism in $\CC(\Z)$ with $d_i \in D$ indecomposable for each $i \in \{1, \ldots, n\}$. 

We can write $\delta = (\delta_1, \ldots, \delta_n)$, and $\delta$ is a $D$-precover of $e$ if and only if each morphism $\varphi \colon d \to e$ with $d \in D$ indecomposable factors through at least one of the $\delta_i$.
\end{Lemma}

\begin{proof}
It is clear that if each morphism $\varphi \colon d \to e$ with $d \in D$ indecomposable factors through at least one of the $\delta_i$, then it also factors through $\delta$ which is hence a $D$-precover. 
	
Conversely, assume that $\delta$ is a $D$-precover.  Let $\varphi \colon d \to e$ be a morphism in $\CC(\Z)$ with $d \in D$ indecomposable.  If $\varphi = 0$, then $\varphi$ factors trivially through each $\delta_i$ and we are done. If $\varphi \neq 0$, then choose a morphism $\varphi' \colon d \to d_1 \oplus \ldots \oplus d_n$ with $\varphi = \delta \circ \varphi'$.  Writing $\varphi'$ in components $\varphi'_i$, this means $\varphi = \delta_1\varphi'_1 + \ldots + \delta_n\varphi'_n$.  Because $\varphi \neq 0$ there exists an $i \in \{1, \ldots, n\}$ such that $\delta_i\circ \varphi'_i \neq 0$. 
Now $\varphi$ and $\delta_i \circ \varphi'_i$ are non-zero elements of $\Hom_{\CC(\Z)}( d,e )$ which must be a one-dimensional $k$-vector space by Section \ref{sec:IT}(vi).  Hence $\varphi = \alpha\, \delta_i\circ \varphi'_i$ for some $\alpha \in k$, so $\varphi$ factors through $\delta_i$.
\end{proof}

\begin{Lemma}
\label{L:precovering condition archaic}
Let $\X$ be a set of diagonals of $\Z$. Then $\add E(\X)$ is a precovering subcategory of $\CC(\Z)$ if and only if $\X$ satisfies the following condition:
	
For each diagonal $Y = \{y_0,y_1\}$ of $\Z$ there is a finite set of diagonals $X^1, \ldots, X^l \in \X$, such that for each $X=\{x_0,x_1\} \in \X$ with 
\[
  x_0 \leqslant y_0 \leqslant x_1^{--} < x_1 \leqslant y_1 \leqslant x_0^{--}
\]
there is an $i \in \{1, \ldots, l\}$ with $X^i = \{x^i_0,x^i_1\}$ and 
\[
  x_0 \leqslant x^i_0 \leqslant y_0 \text{ and } x_1 \leqslant x^i_1 \leqslant y_1.
\]
\end{Lemma}
 
\begin{proof}
This is immediate by combining Section \ref{sec:IT}(vii) with Lemma \ref{L:precover decomposition}.
\end{proof}

\begin{Proposition}
\label{P:precovering implies PC1 and PC2}
Let $\X$ be a set of diagonals of $\Z$. If $\add E(\X)$ is a precovering subcategory of $\cC( \cZ )$ then $\X$ satisfies conditions PC1 and PC2.
\end{Proposition}

\begin{proof}
Let $\X$ be a set of diagonals such that $\add E(\X)$ is precovering. We show that $\X$ satisfies condition PC1. The fact that $\X$ satisfies condition PC2 follows by an analogous argument.  Hence let $X^i = \{x^i_0, x^i_1\}_{i \in \ZZ_{\geqslant 0}}$ be a sequence from $\X$ with $x^i_0 \to p$ from below and $x^i_1 \to q$ from below with $p \neq q$. 

	If $p,q \in \Z$, we have $x^i_0 = p$ and $x^i_1 = q$ from some step, whence $\{p,q\} \in \X$ and condition PC1 is clearly satisfied with $x'^j_0 = p$ and $x'^j_1 = q$ for each $j \in \ZZ_{\geqslant 0}$. 
	
	We can thus assume that $p \in L(\Z)$ or $q \in L(\Z)$.

	 Then by passing to a subsequence we may assume
		\[
			x^i_0 \leqslant p \leqslant {x^i_1}^{--} < x^i_1 \leqslant q \leqslant {x^i_0}^{--}
		\]
	for each $i \in \ZZ_{\geqslant 0}$. Let $Y = \{y_0,y_1\}$ be a diagonal of $\Z$ with
		\[
			p \leqslant y_0 \leqslant {x^i_1}^{--} \text{ and } q \leqslant y_1 \leqslant {x^i_0}^{--}
		\]
	for each $i \in \ZZ_{\geqslant 0}$, see Figure \ref{fig:ys}. 
	Note that such diagonals exist; in fact since $\Z$ satisfies the two-sided
	limit condition (see Definition \ref{D:admissible}), we can even find an entire sequence of such diagonals 
	with endpoints converging to $p$ and $q$ (at least one of which lies in $L(\Z)$) from above.

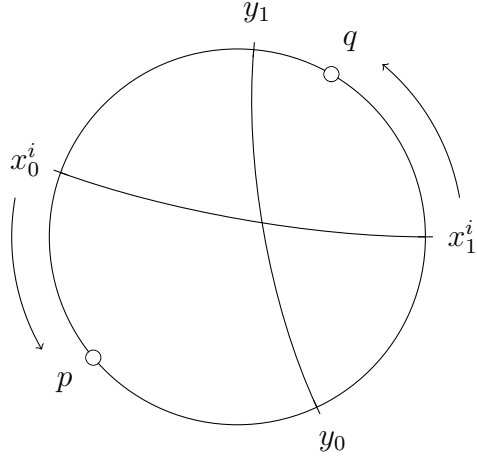
\begin{figure}
	\begin{center}
	\begin{tikzpicture}[scale=5]
        \draw (0,0) circle(0.5cm);
	\draw (60:0.5) node[fill=white,circle,inner sep=0.065cm] {} circle (0.02cm);
	\draw (220:0.5) node[fill=white,circle,inner sep=0.065cm] {} circle (0.02cm);
	\node at (60:0.6){$q$};
	\node at (220:0.6){$p$};

	\node at (85:0.6){$y_1$};
	\draw  (85:0.48) -- (85:0.52);
	\node at (295:0.6){$y_0$};
	\draw  (295:0.48) -- (295:0.52);
    \draw (85:0.5) .. controls (85:0.2) and (295:0.2) .. (295:0.5);	    
	\node at (0:0.6) {$x^i_1$};
	\draw  (0:0.48) -- (0:0.52);
	\node at (160:0.6) {$x^i_0$};
	\draw  (160:0.48) -- (160:0.52);
    \draw (0:0.5) .. controls (0:0.2) and (160:0.2) .. (160:0.5);	    
	\draw[->] (10:0.6) arc (10:50:0.6);
	\draw[->] (170:0.6) arc (170:210:0.6);
	
	\end{tikzpicture}
	\end{center}
\caption{Illustration of the proof of Proposition \ref{P:precovering implies PC1 and PC2}
.} \label{fig:ys}
\end{figure}	
	Then for each $i \in \ZZ_{\geqslant 0}$ we have
		\[
			x^i_0 \leqslant y_0 \leqslant {x^i_1}^{--} \text{ and } x^i_1 \leqslant y_1 \leqslant {x_0^i}^{--}.
		\]
	By assumption, $\add E(\X)$ is a precovering subcategory of $\cC( \cZ )$. So by Lemma \ref{L:precovering condition archaic} there must exist finitely many diagonals $U^j = \{u^j_0, u^j_1\} \in \X$ for $j \in \{1, \ldots, l\}$, such that for each $i \in \ZZ_{\geqslant 0}$ there is a 
	$j \in \{1, \ldots, l\}$ with
		\[
			x^i_0 \leqslant u^j_0 \leqslant y_0 \text{ and } x^i_1 \leqslant u^j_1 \leqslant y_1.
		\]
	There must be a $j \in \{1, \ldots, l\}$ which works for infinitely many values of $i \in \ZZ_{\geqslant 0}$, i.e.\ there is a diagonal $V=\{v_0,v_1\} \in \X$ such that for infinitely many values of
	$i \in \ZZ_{\geqslant 0}$ we have
		\[
			x^i_0 \leqslant v_0 \leqslant y_0 \text{ and } x^i_1 \leqslant v_1 \leqslant y_1. 
		\]
	Since they hold for infinitely many $i \in \ZZ_{\geqslant 0}$, the first of these inequalities forces $p \leqslant v_0 \leqslant y_0$, while the second forces $q \leqslant v_1 \leqslant y_1$. As mentioned above, since $\Z$ satisfies the two-sided limit
	condition, we can pick a sequence of diagonals $Y^j = \{y^j_0,y^j_1\}$ of $\Z$ with $y^j_0 \to p$ from above and $y^j_1 \to q$ from above and such that
		\[
			p \leqslant y^j_0 \leqslant {x^i_1}^{--} \text{ and } q \leqslant y^j_1 \leqslant {x^i_0}^{--}
		\]
	for all $i,j \in \ZZ_{\geqslant 0}$ (note that if $p \in \Z$, respectively $q \in \Z$, we can pick $y^j_0 = p$ for each $j \in \ZZ_{\geqslant 0}$, respectively $y^j_1 = q$ for each $j \in \ZZ_{\geqslant 0}$). Applying the above argument for each of the diagonals $Y^j$ in this sequence, we find a sequence 
	$\{v^j_0,v^j_1\} \in \X$ with $v^j_0\to p$ from above and $v^j_1\to q$ from above. Thus condition PC1 holds. 
\end{proof}

\begin{Remark}
\label{R:special case half fountain}
Either of conditions PC1 and PC2 implies the following condition: Suppose $\X$ has a right fountain at $z \in \Z$ converging to $a \in L(\Z)$, that is, a sequence $\{z,x_i\}_{i\in\ZZ_{\geqslant 0}}$ with $x_i\to a$ from below.  Then $\X$ has a fountain at $z$ converging to $a$.
	
Namely, if condition PC1 holds, then there is a sequence $\{x'^i_0,x'^i_1\} _{i\in\ZZ_{\geqslant 0}}$ from $\X$ with $x'^i_0\to z$ from above and $x'^i_1\to a$ from above. Since $\Z$ is discrete, $x'^i_0=z$ from some step (see Remark \ref{R:dichotomy}), so $\X$ has a left fountain at $z$ converging to $a$. 
	
If condition PC2 holds, the analogous argument works with $z$ in the role of $q$ and $a$ in the role of $p$ in the definition of condition PC2.
\end{Remark}
\medskip

\begin{Definition}
\label{N:some sets}
Let $\X$ be a set of diagonals of $\Z$, let $Y = \{y_0,y_1\}$ be in $\X$, and let $t_0 \in [y_1^{++},y_0] \cap \Z$ and $t_1 \in [y_0^{++} ,y_1] \cap \Z$. We write
\[
  W_0(\X,Y, t_0,t_1)
  = \big\{x_0 \in [y_1^{++}, t_0] \cap \Z \,\big|\, \exists \ \{x_0,x_1\} \in \X \text{ with } x_1 \in [t_1,y_1] \big\}.
\]
For $u_0 \in [y_1^{++}, y_0] \cap \Z$ we write
\[
  W_1(\X,Y,u_0,t_1) = \big\{ x_1 \in [t_1,y_1] \cap \Z \,\big|\, \{u_0,x_1\} \in \X \big\}.
\]
\end{Definition}
\medskip
	
The set $W_0(\X,Y, t_0,t_1)$ consists of the end points in $[y_1^{++}, t_0]$ of diagonals in $\X$ between the two intervals shown in Figure \ref{fig:W0}. The set $W_1(\X,Y,u_0,t_1)$ consists of end points in $[t_1,y_1]$ of diagonals of $\X$ with other end point $u_0$.
\begin{figure}
  \centering
    \begin{tikzpicture}[scale=2.5]
      \draw (0,0) circle (1cm);

      \draw[very thick] ([shift=(-100:1cm)]0,0) arc (-100:55:1cm);      
      \draw[very thick] ([shift=(145:1cm)]0,0) arc (145:235:1cm);            

      \draw (-100:0.97cm) -- (-100:1.03cm);
      \draw (-100:1.13cm) node{$y_1^{++}$};
      \draw (-50:0.97cm) -- (-50:1.03cm);
      \draw (-50:1.13cm) node{$x_0$};
      \draw (-30:0.97cm) -- (-30:1.03cm);
      \draw (-30:1.13cm) node{$x'_0$};
      \draw (-5:0.97cm) -- (-5:1.03cm);
      \draw (-5:1.13cm) node{$x''_0$};
      \draw (25:0.97cm) -- (25:1.03cm);
      \draw (25:1.13cm) node{$s_0$};
      \draw (55:0.97cm) -- (55:1.03cm);
      \draw (55:1.13cm) node{$t_0$};
      \draw (95:0.97cm) -- (95:1.03cm);
      \draw (95:1.13cm) node{$y_0$};      
      \draw (120:0.97cm) -- (120:1.03cm);
      \draw (120:1.13cm) node{$y_0^{++}$};      
      \draw (145:0.97cm) -- (145:1.03cm);
      \draw (145:1.13cm) node{$t_1$};
      \draw (155:0.97cm) -- (155:1.03cm);
      \draw (155:1.13cm) node{$x_1$};
      \draw (170:0.97cm) -- (170:1.03cm);
      \draw (170:1.13cm) node{$x'_1$};
      \draw (188:0.97cm) -- (188:1.03cm);
      \draw (188:1.13cm) node{$x''_1$};
      \draw (210:0.97cm) -- (210:1.03cm);
      \draw (210:1.13cm) node{$s_1$};
      \draw (235:0.97cm) -- (235:1.03cm);
      \draw (235:1.15cm) node{$y_1$};

      \draw (25:1cm) .. controls (25:0.5cm) and (210:0.5cm) .. (210:1cm);
      \draw (-50:1cm) .. controls (-50:0.50cm) and (155:0.50cm) .. (155:1cm);
      \draw (-30:1cm) .. controls (-30:0.50cm) and (170:0.50cm) .. (170:1cm);
      \draw (-5:1cm) .. controls (-5:0.5cm) and (188:0.5cm) .. (188:1cm);
      \draw (95:1cm) .. controls (95:0.3cm) and (235:0.3cm) .. (235:1cm);

    \end{tikzpicture} 
  \caption{Illustration of Definition \ref{N:some sets}.  The set $W_0(\X,Y, t_0,t_1)$ has elements $x_0$, $x_0'$, and $x_0''$ among others, and supremum $s_0$, where $\{x_0,x_1\}$, $\{x'_0,x'_1\}$, $\{x''_0,x''_1\}$, $\{s_0,s_1\}$ are diagonals in $\X$.}
\label{fig:W0}
\end{figure}
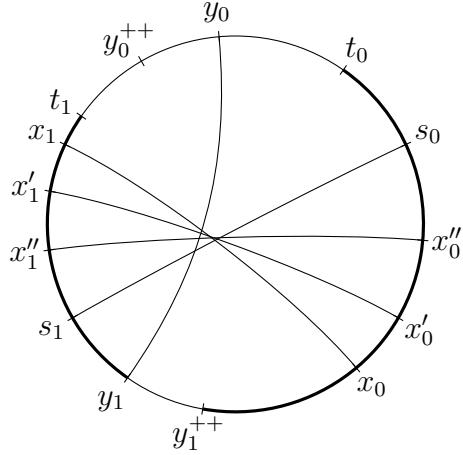

\begin{Lemma}
\label{L:supremum general}
Let $\X$ be a set of diagonals of $\Z$ satisfying conditions PC1 and PC2, let $Y = \{y_0,y_1\}$ be in $\X$, and let $t_0, t_1$ and $u_0$ be as in Definition \ref{N:some sets}.  Then the following holds.
\begin{enumerate}
\setlength\itemsep{4pt}

  \item  {If the set $W_0:=W_0(\X,Y,t_0,t_1)$ is non-empty, then $s_0 := \sup_{[y_1^{++}, t_0]} W_0 \in \Z$.}

  \item  {If the set $W_1:=W_1(\X,Y,u_0,t_1)$ is non-empty, then $s_1 := \sup_{[t_1,y_1]} W_1 \in \Z$.}

\end{enumerate}
\end{Lemma}
	
\begin{proof}
We start by showing (i). Suppose $s_0 \notin \Z$, in particular $s_0 \neq y_1^{++}$ and $s_0 \neq t_0$, so $s_0 \in (y_1^{++}, t_0)$. There is a sequence $\{x^i_0,x^i_1\}$ from $\X$ with $x^i_0 \in [y_1^{++},t_0]$, $x^i_1 \in [t_1,y_1]$ for each $i \in \ZZ_{\geqslant 0}$ and $x^i_0 \to s_0$ from below. Passing to a subsequence we can assume $x^i_1 \to \tilde{s_1}$ from below or above for some $\tilde{s_1} \in [t_1,y_1]$. Note that since
\[
  y_1^{++} < s_0 < t_0 \leqslant y_0 < y_0^{++} \leqslant t_1 \leqslant \tilde{s_1} \leqslant y_1,
\]
we have $s_0 \neq \tilde{s_1}$. So conditions PC1 and PC2 imply that there is a sequence $\{x'^i_0,x'^i_1\}$ from $\X$ with $x'^i_0 \to s_0$ and $x'^i_1 \to \tilde{s_1}$ both from above.  So for some $i \in \ZZ_{\geqslant 0}$ we have 
\[
  s_0 < x'^i_0 \leqslant t_0 \text{ and } \tilde{s_1} \leqslant x'^i_1 \leqslant y_1.
\]
In particular, $x'^i_0 \in [y_1^{++}, t_0]$ and $x'^i_1 \in [t_1,y_1]$ for these $i$, so $x'^i_0 \in W_0$ and the first of the above inequalities violates the definition of $s_0$ as a supremum.
		
We now show (ii). Suppose $s_1 \notin \Z$, in particular $s_1\neq t_1$ and $s_1 \neq y_1$, so $s_1 \in (t_1,y_1)$.  There is a sequence $\{u_0,x^i\}$ from $\X$ with $x^i \in [t_1,y_1]$ for each $i \in \ZZ_{\geqslant 0}$ and $x^i \to s_1$ from below. By condition PC1 (or PC2) and Remark \ref{R:special case half fountain} there is a sequence $\{u_0,x'^i\}$ from $\X$ with $x'^i \to s_1$ from above. However, then we obtain $s_1 < x'^i \leqslant y_1$ from some step, violating the definition of $s_1$ as a supremum.
\end{proof}

We can now prove Theorem \ref{T:precovering}.

\begin{proof}
If $\add E(\X)$ is precovering, then $\X$ satisfies conditions PC1 and PC2 by Proposition \ref{P:precovering implies PC1 and PC2}.

Conversely, assume that $\X$ satisfies conditions PC1 and PC2. Let $Y = \{y_0,y_1\}$ be an arbitrary diagonal of $\Z$.  According to Lemma \ref{L:precovering condition archaic} we have to show that $Y$ satisfies the following condition:
		\begin{itemize}
		\setlength\itemsep{4pt}
		\item[($\ast$)]{ There exists a finite set of diagonals $S=\{X^1, \ldots, X^l\} \subseteq \X$, such that for each
		diagonal $X=\{x_0,x_1\} \in \X$ with 
			$x_0 \leqslant y_0 \leqslant x_1^{--}$ and $x_1 \leqslant y_1 \leqslant x_0^{--}$
		there is an $i \in \{1, \ldots, l\}$ with $X^i = \{x^i_0,x^i_1\}$ and 
			$x_0 \leqslant x^i_0 \leqslant y_0$ and $x_1 \leqslant x^i_1 \leqslant y_1.$}
		\end{itemize}
		
		We are going to construct inductively a sequence $S$ of diagonals from $\X$, see Figure \ref{fig:ss}.
		
\begin{figure}
		\begin{center}
	\begin{tikzpicture}[scale=5,cap=round,>=latex]
        \draw (0,0) circle(0.5cm);
	\node at (350:0.6cm){$y_0$};
	\draw (350:0.48) -- (350:0.52);
	\node at (5:0.8cm){$y_0^+ = s_0^0=s_1^0$};
	\draw (5:0.48) -- (5:0.52);
	\node at (20:0.64cm) {$y_0^{++}$};
	\draw (20:0.48) -- (20:0.52);
	\node at (35:0.6) {$s_1^1$};
	\draw (35:0.48) -- (35:0.52);
	\node at (50:0.6) {$s_1^2$};
	\draw (50:0.48) -- (50:0.52);
	\node at (70:0.6) {$s_1^3$};
	\draw (70:0.48) -- (70:0.52);
	\draw[thick,dotted] (80:0.58) arc (80:90:0.58);
	\node at (140:0.6) {$y_1$};
	\draw (140:0.48) -- (140:0.52);
	\node at (160:0.6) {$y_1^+$};
	\draw (160:0.48) -- (160:0.52);
	\node at (180:0.6) {$y_1^{++}$};
	\draw (180:0.48) -- (180:0.52);
	\draw[thick,dotted] (260:0.58) arc (260:270:0.58);
	\node at (280:0.6) {$s_0^3$};
	\draw (280:0.48) -- (280:0.52);
	\node at (310:0.6) {$s_0^2$};
	\draw (310:0.48) -- (310:0.52);
	\node at (340:0.6) {$s_0^1$};
	\draw (340:0.48) -- (340:0.52);
	
    \draw (140:0.5) .. controls (140:0.2) and (350:0.2) .. (350:0.5);	    
    \draw (35:0.5) .. controls (15:0.35) and (355:0.35) .. (340:0.5);	    
    \draw (50:0.5) .. controls (50:0.25) and (310:0.25) .. (310:0.5);	    
    \draw (70:0.5) to[bend right] (280:0.5);

	\draw[thick,dotted] (-0.15,0) -- (-0.05,0);
       
      \end{tikzpicture}
      \end{center}
  \caption{Illustration of the proof of Theorem \ref{T:precovering}.}
\label{fig:ss}
\end{figure}
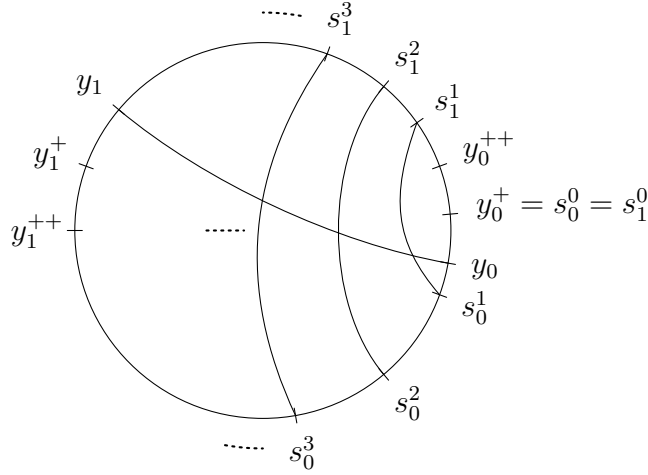

		Set $s^0_0 = s^0_1 = y_0^+$. For $l \geqslant 1$, if
			\[
				y_1^{++} \leqslant s_0^{l-1} \leqslant y_0^+ \text{ and } y_0^+ \leqslant s^{l-1}_1 \leqslant y_1
			\]
		have already been defined, then we proceed as follows: 
			\begin{itemize}
			\setlength\itemsep{4pt}
				\item
					{If $s^{l-1}_0 = y_1^{++}$ or $s^{l-1}_1 = y_1$, then we terminate.} (Note that for $l=1$ 
					this can not happen since $\{y_0,y_1\}$ is a diagonal, i.e. $y_0$ and $y_1$ are not neighbouring 
					vertices of $\Z$.)
				\item
					{If $s^{l-1}_0 \neq y_1^{++}$ and $s^{l-1}_1 \neq y_1$, then 
						\[
							y_1^{++} \leqslant (s^{l-1}_0)^- \leqslant y_0 \text{ and } y_0^{++} \leqslant (s^{l-1}_1)^+ \leqslant y_1
						\]
					and we set $t_0 = (s^{l-1}_0)^-$, $t_1 = (s^{l-1}_1)^+$. 
					
					If $W_0(\X,Y,t_0,t_1) = \varnothing$ then we terminate. (Note that if this happens for $l=1$
					then there are no relevant diagonals $X$ as in condition ($\ast$), thus ($\ast$) is trivially
					satisfied.)
					
					If $W_0(\X,Y,t_0,t_1) \neq \varnothing$
					then we set 
						\begin{eqnarray} \label{Eqn:sl0}
							s^l_0 = \sup_{[y_1^{++}, t_0]} W_0(\X,Y,t_0,t_1).
						\end{eqnarray}
					This supremum lies in $\Z$ by Lemma \ref{L:supremum general}.  We then set $u_0 = s^l_0$ and consider the set
					$W_1(\X,Y,u_0,t_1)$.  It is non-empty since $W_0(\X,Y,t_0,t_1) \neq \varnothing$, and we set 
						\[
							s^l_1 = \sup_{[t_1,y_1]} W_1(\X,Y,u_0,t_1).
						\] 
					}
			\end{itemize}
      
		Note that by construction we have
				\begin{eqnarray}\label{Eqn:s^i_0}
					y_1^{++}  \leqslant \ldots < s^3_0 < s^2_0 < s^1_0 \leqslant y_0,
				\end{eqnarray}
				\begin{eqnarray}\label{Eqn:s^i_1}
					y_0^{++} \leqslant s^1_1 < s^2_1 < s^3_1 < \ldots \leqslant y_1,
				\end{eqnarray}
		and $\{s^l_0,s^l_1\} \in \X$ by Lemma \ref{L:supremum general}
		for all $l \geqslant 0$ that are defined. 
				
		We now show that 
		our construction terminates after finitely many steps for each diagonal $Y$ of $\Z$. 
		Suppose by contradiction that for some diagonal $Y = \{y_0,y_1\}$ of $\Z$, our construction does not terminate.
		Then by the inequalities (\ref{Eqn:s^i_0}) and (\ref{Eqn:s^i_1}) there must exist $a \in (y_1^{++}, y_0) \cap L(\Z)$ and $b \in (y_0^{++}, y_1) \cap L(\Z)$ such that $s^l_0 \to a$ from above and $s^l_1 \to b$ from
		below. By condition PC2 for $\X$, there is a sequence $\{s'^m_0,s'^m_1\}$ from $\X$ 
		such that $s'^m_0 \to a$ from above and $s'^m_1 \to b$ from above. 
		Moreover, there exist $m,l \in \ZZ_{\geqslant 0}$ such that
		$s^l_0 < s'^m_0 \leqslant s^{l-1}_0$ and $b < s'^m_1 < y_1$. If we have ${s'}^m_0 = s^{l-1}_0$ then $\{s_0^{l-1},{s'}_1^m\} \in \X$ contradicts the definition of $s_1^{l-1}$ as a supremum. Else, since we now have $(s^{l-1}_1)^+ < b < s'^m_1 < y_1$, the diagonal $\{s'^m_0,s'^m_1\} \in \X$ violates the definition of $s^l_0$ as a supremum. 
				
		So we have shown that our construction terminates after finitely many steps. By the above remarks on
		the case $l=1$ (i.e. that if the construction terminates without defining $s_0^1$ and $s_1^1$ then condition
		($\ast$) is trivially satisfied) we can assume that the construction provides a non-empty finite set
\[
  S = \big\{\{s^l_0, s^l_1\} \,\big|\, 1 \leqslant l \leqslant N \big\}
\]
of diagonals from $\X$, for some $N\in\mathbb{N}$. 
		 
		We now finally show that the set $S$ has the desired 
		property from condition ($\ast$). Let $X=\{x_0,x_1\}\in \X$ with $x_0\leqslant y_0\leqslant x_1^{--}$
		and $x_1\leqslant y_1\leqslant x_0^{--}$, i.e. $x_0\in [y_1^{++},y_0]$ and $x_1\in [y_0^{++},y_1]$.
		
		We distinguish two cases. 
		Assume first that there is an $l \geqslant 1$ such that $s^l_0 < x_0 \leqslant s^{l-1}_0$. Note that then
		$l \geqslant 2$ since for $l=1$ this would violate the definition of $s_0^1$ as supremum.
		Recall from equation (\ref{Eqn:sl0}) that
			$$
				s^l_0 = \sup_{[y_1^{++},(s^{l-1}_0)^-]} W_{0}(\X,Y,(s^{l-1}_0)^-,(s^{l-1}_1)^+),
			$$
		so $s^l_0 < x_0 \leqslant s^{l-1}_0$ implies that there is no diagonal $\{x_0,v_1\} \in \X$ with $v_1 \in [(s^{l-1}_1)^+,y_1]$. That is, we must have $y_0^{++} \leqslant x_1 \leqslant s^{l-1}_1$.
		We get that $x_0 \leqslant s^{l-1}_0 \leqslant y_0$ and $x_1 \leqslant s^{l-1}_1 \leqslant y_1$, so we are done in this case.
		
		Assume now that there is no $l \geqslant 1$ such that $s^l_0 < x_0 \leqslant s^{l-1}_0$. This means that 
		$x_0\in [y_1^{++},s_0^N]$. Since $s_0^{N+1}$ has not been defined in our construction and by the 
		choice of $s_1^N$ as supremum we must have $x_1\in [y_0^{++},s_1^N]$. In other words, 
		$x_0\leqslant s_0^N\leqslant y_0$ and $x_1\leqslant s_1^N\leqslant y_1$, and hence condition ($\ast$) is also satisfied in this case. 
	\end{proof}

\section{Torsion pairs in the cluster categories $\CC(\Z)$}
\label{sec:torsion_pairs}

This section proves Theorem \ref{thm:A} from the introduction (=Theorem \ref{T:torsion pair}).  To set the scene, recall the definition of torsion pairs in triangulated categories, due to Iyama and Yoshino \cite[def.\ 2.2]{IY:mutation}, following the lead of Dickson \cite[p.\ 224]{Dickson:torsiontheory} from the abelian case.

\begin{Definition}
[Torsion pairs in triangulated categories]
\label{def:torsion_pairs}
Let $\T$ be a triangulated category with suspension functor $\Sigma$. A pair $(X,Y)$ of full subcategories of $\T$ is called a {\em torsion pair} if it satisfies the following two axioms.
	
\begin{VarDescription}{(T2)\quad}
\setlength\itemsep{4pt}

			\item[(T1)]{$\Hom_\T(x,y) = 0$ for all $x \in X$, $y \in Y$.}

			\item[(T2)]{For each $t \in \T$ there exist $x \in X$ and $y \in Y$ and a distinguished triangle
						\[
							x \to t \to y \to \Sigma x.
						\]}
\end{VarDescription}
\end{Definition}
\medskip

\begin{Lemma}\label{L:set U}
	Let $\X$ be a set of diagonals of $\Z$ satisfying condition PC1 or condition PC2 and 
	let $s,t \in \Z$. If the set
		\[
			U([s,t]) = \{z \in [s,t] \cap \Z \mid \{s,z\} \in \X \}
		\]
	is non-empty then its supremum $u = \sup_{[s,t]} U([s,t])$
	lies in $\Z$.
\end{Lemma}

\begin{proof}
	Assume by contradiction that the supremum $u$ does not lie in $\Z$. Then there is a sequence $\{s,z^i\}_{i \in \ZZ_{\geqslant 0}}$ from $\X$ with $z^i \to u$ from below. Since $\X$ satisfies condition PC1 or condition PC2, by Remark \ref{R:special case half fountain} there is a sequence $\{s,z'^i\}_{i \in \ZZ_{\geqslant 0}}$ from $\X$ with $z'^i \to u$ from above. Since 
	$u \notin \Z$ we have $u \neq t$ and thus $u < z'^i < t$ for some $i \in \ZZ_{\geqslant 0}$. Then $\{s,z'^i\} \in \X$ violates the definition of $u$ as a supremum.
\end{proof}

\begin{Lemma}
\label{L:16}
Let $\X$ be a set of diagonals of $\Z$ satisfying conditions PC1 and PC2 and the Ptolemy condition.  Let $s\in \Z$ and $v \in L(\Z)$ be given, and assume that there exists $t\in (s,v) \cap \Z$ such that the following condition is satisfied, see Figure \ref{fig:L:16}:
\begin{equation}
\label{eqn:s<t<v}
  \text{For each $w \in (t,v) \cap \Z$ there exists a diagonal $\{p,q\} \in \X$ with} \; s < p < w < q < v.
\end{equation}
Then for each $w \in (t,v) \cap \Z$ there exists a diagonal $\{p',q'\} \in \X$ with
\[
  s < p' \leqslant t < w < q' < v.
\]
\end{Lemma}

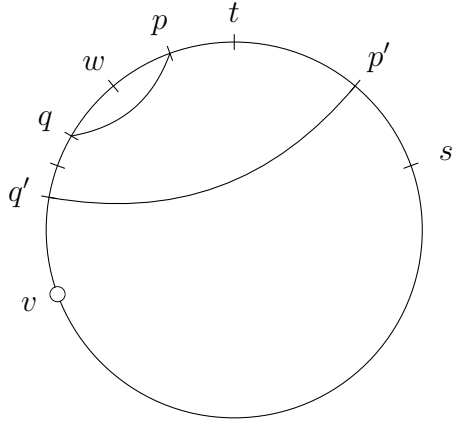
\begin{figure}
\centering{
\begin{tikzpicture}[scale=5,cap=round,>=latex]
\draw (0,0) circle(0.5cm);
	\node (d) at (20:0.6cm){$s$};
	\draw (20:0.48) -- (20:0.52);
	\draw (200:0.5) node[fill=white,circle,inner sep=0.065cm] {} circle (0.02cm);	
	\node (d) at (200:0.58cm){$v$};
	\node (d) at (50:0.60cm){$p'$};
	\draw (50:0.48) -- (50:0.52);
	\node (d) at (90:0.58cm){$t$};
	\draw (90:0.48) -- (90:0.52);
	\node (d) at (110:0.58cm){$p$};
	\draw (110:0.48) -- (110:0.52);
	\node (d) at (130:0.58cm){$w$};
	\draw (130:0.48) -- (130:0.52);
	\node (d) at (150:0.58cm){$q$};
	\draw (150:0.48) -- (150:0.52);
	\node (d) at (170:0.58cm){$q'$};
	\draw (170:0.48) -- (170:0.52);

	\draw (160:0.48) -- (160:0.52);
	\draw (50:0.5) to[bend left] (170:0.5);
	\draw (110:0.5) to[bend left] (150:0.5);
        	
\end{tikzpicture}
\caption{Illustration of Lemma \ref{L:16}.} \label{fig:L:16}
}
\end{figure}

\begin{proof}
Consider the set
\[
  V = \big\{ w \in (t,v) \cap \Z \,\big|\, \nexists \{p',q'\} \in \X \text{ with } s < p' \leqslant t < w < q' < v \big\}
\]
and suppose that $V \neq \varnothing$, setting $\tilde{w} = \inf_{[t,v]} V$. We aim for a contradiction.
	
	Assume first that $\tilde{w} \in \Z$. In particular, this implies $\tilde{w} \in V$. By condition (\ref{eqn:s<t<v}) there exists a diagonal $\{p,q\} \in \X$ with 
		\[
			s < p < \tilde{w} < q < v.
		\]
	Since $\tilde{w} \in V$ we must have
		\begin{equation}\label{Inequality 1}
			s < t < p < \tilde{w} < q < v.
		\end{equation}
	Therefore $\tilde{w}$ lies in $(t^{+},v)$ and thus $\tilde{w}^- \in (t,v)$. Now, because $\tilde{w}$ is the infimum of $V$, we have $\tilde{w}^- \notin V$ and thus we can find $\{p',q'\} \in \X$ with $s < p' \leqslant t < \tilde{w}^- < q' < v$. 
	
	This implies
	\begin{equation}\label{Inequality 2}
		s < p' \leqslant t < \tilde{w} \leqslant q' < v
	\end{equation}
	and because $\tilde{w} \in V$
	we must have $q' = \tilde{w}$. 
	Combining (\ref{Inequality 1}) and (\ref{Inequality 2}) yields
		\[
			s <  p' \leqslant t < p < \tilde{w} = q' < q <v,
		\]
	implying that $\{p',q'\} \in \X$ and $\{p,q\} \in \X$ cross.
	The Ptolemy condition implies that the diagonal
	$\{p',q\}$ is in $\X$ and we have
		\[
			s < p' \leqslant t < \tilde{w} < q < v.
		\]
	This contradicts $\tilde{w} \in V$.
	
	Assume now that $\tilde{w} \in L(\Z)$. We can pick a sequence $\{w^i\}_{i \in \mathbb{Z}_{\geqslant 0}}$ 
	from $V$ converging to $\tilde{w}$ from above. 
	Since $\Z$ satisfies the two-sided limit condition (cf. Definition \ref{D:admissible}), 
	we can pick a sequence $\{z^i\}_{i \in \mathbb{Z}_{\geqslant 0}}$ from $(t, \tilde{w})\cap \Z$ 
	converging to $\tilde{w}$ from below.
	
	Because $\tilde{w}$ is the infimum of $V$, we have $z^i \notin V$ for each $i \in \ZZ_{\geqslant 0}$. Thus for each $i \in \ZZ_{\geqslant 0}$ there is a diagonal $\{x^i_0, x^i_1\} \in \X$ with 
		\[
			s < x^i_0 \leqslant t < z^i < x^i_1 < v.
		\]
	The last inequality can even be written $x^i_1 < \tilde{w} < v$ for each $i \in \ZZ_{\geqslant 0}$: 
	If we had $\tilde{w} < x^i_1 < v$ for an $i \in \ZZ_{\geqslant 0}$ there would be a $j \in \ZZ_{\geqslant0}$ 
	(in fact, infinitely many) with $\tilde{w} < w^j < x_1^i$ which would yield
	\[
		s < x_0^i \leqslant t < \tilde{w} < w^j < x_1^i < v
	\]
	contradicting the fact that $w^j \in V$.
		
	Having $z^i < x^i_1 < \tilde{w}$ for each $i \in \ZZ_{\geqslant 0}$ and $z^i \to \tilde{w}$ from below 
	forces $x^i_1 \to \tilde{w}$ from below. 
	We have $x^i_0 \in [s^+,t]$ and passing to a subsequence we can assume $x^i_0 \to c$ from below or above for some 
	$c \in [s^+,t] \cap \overline{\Z}$. Since $\tilde{w} \in (t,v) \cap L(\Z)$ we have $c \neq \tilde{w}$. 
	By assumption, the set $\X$ satisfies conditions PC1 and PC2 and thus  
	there is a sequence $\{x'^i_0,x'^i_1\}_{i \in \ZZ_{\geqslant 0}}$ of diagonals from $\X$ with $x'^i_0 \to c$ from above 
	and $x'^i_1 \to \tilde{w}$ from above. We can pick $i,j \in \mathbb{Z}_{\geqslant 0}$ such that
	\[
		s < x'^i_0 \leqslant t < \tilde{w} < w^j < x'^i_1<v
	\]
	contradicting the fact that $w^j \in V$.
	\end{proof}

\begin{Definition}
\label{def:nc}
Let $\X$ be a set of diagonals of $\Z$. Then we set
\[
  \nc \X = \{Y \text{ diagonal of } \Z \mid Y \text{ crosses no } X \in \X\}.
\]
We write $\nc^2 \X = \nc (\nc \X)$.  The letters ``nc'' stand for ``non-crossing''.
\end{Definition}
\medskip

\begin{Lemma}\label{L:nc^2 = id implies Ptolemy}
	Let $\X$ be a set of diagonals of $\Z$. If $\nc^2 \X = \X$, then $\X$ satisfies the Ptolemy condition.
\end{Lemma}

\begin{proof}
Assume $\{x_0,x_1\} \in \X$ and $\{y_0,y_1\} \in \X$ cross. According to Definition \ref{def:diagonals} this means that we can label the vertices so that $x_0<y_0<x_1<y_1$. Consider those of $\{x_0,y_0\}$, $\{y_0,x_1\}$, $\{x_1,y_1\}$ and $\{y_1,x_0\}$ which are diagonals of $\Z$.  Clearly, any diagonal $U$ of $\Z$ crossing one of these diagonals must also cross one of $\{x_0,x_1\} \in \X$ and $\{y_0,y_1\} \in \X$, i.e. $U\not\in \nc\X$. It follows that those of $\{x_0,y_0\}$, $\{y_0,x_1\}$, $\{x_1,y_1\}$ and $\{y_1,x_0\}$ which are diagonals of $\Z$ lie in $\nc^2 \X$. But by assumption $\nc^2 \X = \X$, so $\X$ satisfies the Ptolemy condition.
\end{proof}

\begin{Lemma}\label{L:Ptolemy implies nc^2 = id}
	Let $\X$ be a set of diagonals of $\Z$ satisfying conditions PC1 and PC2.  
	If $\X$ satisfies the Ptolemy condition, then $\nc^2 \X = \X$.
\end{Lemma}

\begin{proof}
	The inclusion $\X \subseteq \nc^2 \X$ follows immediately from Definition \ref{def:nc}
	(and does not	need any of the assumptions on $\X$).
	  
	For the inclusion $\nc^2 \cX \subseteq \cX$, let $\{s,t\} \in \nc^2 \X$ be given. Our proof will
	be divided into cases and subcases. 
	For each one we will show either that $\{s,t\} \in \X$, or that we
	can deduce a contradiction.
	
	{\em Case A: There does not exist $z \in (s,t] \cap \Z$ such that $\{s,z\} \in \X$.} We will show that 
	this assumption leads to a contradiction. 
	
	Observe that $\{s,t\} \in \nc^2 \X$ implies $\{s^-,s^+\} \notin \nc \X$, so there exists a $z \in \Z$ such that $\{s,z\} \in \X$. By assumption we have 
	$z \notin (s,t]$, so the set
		\[
			V = \big\{z \in (t,s) \cap \Z \,\big|\, \{s,z\} \in \X \big\}
		\]
	is non-empty. Set $v = \inf_{(t,s)}V$.
	We claim that $v\in L(\Z)$. 
	Assume for a contradiction that $v \in \Z$. Then we have $\{s,v\} \in \X$. It follows from the assumption in Case A that $\{s,t\} \notin \X$, so $v \in [t^+,s^{--}]$. Then 
	$\{s^+,v\}$ crosses $\{s,t\} \in \nc^2 \X$, whence 
	$\{s^+,v\} \notin \nc \X$. Thus there is a diagonal $\{p,q\} \in \X$ crossing $\{s^+,v\}$. 
	However, this diagonal can not have $s$ as one of its endpoints, due to the assumption in Case A
	and the definition of $v$ as infimum. So we can deduce that the diagonal $\{p,q\}\in \X$ crosses
	the diagonal $\{s,v\}\in \X$; in particular, one of the endpoints, say $p$, lies in $(s,v)$. But then
	the Ptolemy condition yields that $\{s,p\}\in \X$, contradicting the assumption in Case A 
	and the definition of $v$ as an infimum.
		
	We thus have shown that 
	$v \in L(\Z)$ with $t < v < s$. From the definition of $v$ as infimum there must exist	
	a sequence of diagonals $\{s,v_i\}_{i \in \ZZ_{\geqslant 0}}$ from $\X$ with $v_i \in (v,s)$ and 
	$v_i \to v$ converging from above. Since $\Z$ satisfies the two-sided limit condition
	(see Definition \ref{D:admissible}), there is also a sequence of points in $\Z$ converging to 
	$v$ from below; in particular, $(t,v) \cap \Z$ is non-empty.
	
	For each such $w \in (t,v) \cap \Z$ we have $s^+ < t < w < v < s$,
	so $\{s,t\} \in \nc^2 \X$ crosses $\{s^+,w\}$ whence $\{s^+,w\} \notin \nc \X$. So there is 
	a diagonal $\{p,q\} \in \X$ crossing $\{s^+,w\}$. 
	This diagonal cannot have $s$ as one of its endpoints
	because of the assumption in Case A and the definition of $v$ as infimum.
	So we can assume $p \in [s^{++},w^-]$ and $q \in [w^+,s)$. 
	If $v < q < s$ then there exists an $i \in \ZZ_{\geqslant 0}$ 
	such that $\{s,v_i\}\in \X$ and $ \{p,q\} \in \X$ cross; by the Ptolemy condition it follows
	that $\{s,p\} \in \X$, contradicting our assumption in Case A and the definition of $v$ as infimum.
	
	Since this argument worked for each $w\in (t,v)\cap \Z$, we can apply Lemma \ref{L:16}.
	Thus for each $w\in (t,v)\cap \Z$ there exists a diagonal $\{p',q'\} \in \X$ with
		\begin{equation}\label{Eqn:star}
			s^+ < p' \leqslant t < w < q' < v.
		\end{equation}
		
		As already mentioned above, the two-sided limit condition yields a sequence 
		$\{w_i\}_{i\in\ZZ_{\geqslant 0}}$ with $w_i\to v$ from below.
	By (\ref{Eqn:star}) we can find a sequence $\{p'_i,q'_i\}_{i \in \ZZ_{\geqslant 0}}$ of diagonals
	from $\X$ with $p'_i \in [s^{++},t]$ and 
	$q'_i \in (w_i,v)$ for each $i \in \ZZ_{\geqslant 0}$. It is clear that $q'_i \to v$ from below and by compactness and passing to a subsequence we can assume $p'_i \to r$ from below or above for some $r \in [s^{++},t]$.
	By conditions PC1 and PC2 there is also a sequence $\{p''_i,q''_i\}_{i \in \ZZ_{\geqslant 0}}$ from $\X$ with $p''_i \to r$ and $q''_j \to v$ from above. 
	This implies that there must exist $j, l \in \ZZ_{\geqslant 0}$ such that
	$\{s,v_l\} \in \X$ and $\{p''_j,q''_j\} \in \X$ cross (more precisely, for each $j$ there are
	infinitely many $l$ such that $\{s,v_l\} \in \X$ and $\{p''_j,q''_j\} \in \X$ cross).
	But then the Ptolemy condition gives $\{s,p''_j\} \in \X$, contradicting the assumption in 
	Case A. 
	
	Therefore we have now shown that Case A cannot occur.
	
	{\em Case B: There exists a $z \in (s,t] \cap \Z$ such that $\{s,z\} \in \X$.}  Then the set $U([s,t])$ from Lemma \ref{L:set U} is non-empty, and by Lemma \ref{L:set U} its supremum $u = \sup_{[s,t]}U([s,t])$ lies in $\Z$.
	
	{\em Subcase B1: We have $u = t$. } Then $\{s,t\} = \{s,u\} \in \X$ and we are done. 
	
	{\em Subcase B2: We have $u \in (s,t)$. } 	We will show that this assumption also leads to a contradiction.

	Again, consider the set
		\[
			V = \big\{ y \in (t,s) \cap \Z \,\big|\, \{s,y\} \in \X \big\}.
		\]
		If $V=\emptyset$ then a symmetric version of the assumption in Case A is satisfied; 
		so we can deduce a contradiction exactly as in Case A.
	So we can assume that $V \neq \emptyset$. Set 
		$v = \inf_{(t,s)}V$. 
	
	First suppose $v \in \Z$. Then $\{s,v\} \in \X$ and $t<v<s$. Since $s<u<t$ we have that $\{u,v\}$ is a diagonal of $\Z$ which crosses $\{s,t\}$. Since
	$\{s,t\} \in \nc^2 \X$, this means that $\{u,v\} \notin \nc \X$. So there is a diagonal 
	$\{p,q\} \in \X$ which crosses $\{u,v\}$ and we can assume $p \in [u^+,v^-]$ and $q \in [v^+,u^-]$.
	
	Note that $q=s$ is impossible due to the definition of $u$ as supremum and of $v$ as 
	infimum, respectively. Thus we have $q\neq s$; but then the diagonal $\{p,q\}\in \X$
	crosses $\{s,u\}\in \X$ or $\{s,v\}\in\X$. In either case, the Ptolemy condition implies
	that $\{s,p\}\in\X$, again contradicting the choice of $u$ and $v$ as supremum and infimum,
	because $p\in[u^+,v^-]$.
	
	Therefore we suppose now that $v \notin \Z$, so we have $v \in L(\Z) \cap (t,s)$ as sketched in Figure \ref{fig:sutwv}.
\begin{figure}
	\begin{center}
	 \begin{tikzpicture}[scale=5,cap=round,>=latex]
	 \draw (0,0) circle(0.5cm);
	\node (d) at (0:0.6cm){$s$};
	\draw (0:0.48) -- (0:0.52);
	\draw (-20:0.48) -- (-20:0.52);
	\draw (-40:0.48) -- (-40:0.52);
	\draw (-60:0.48) -- (-60:0.52);	
	\draw (275:0.5) node[fill=white,circle,inner sep=0.065cm] {} circle (0.02cm);	
	\node (d) at (270:0.6cm){$v$};
	\node (d) at (110:0.6cm){$u$};
	\draw (110:0.48) -- (110:0.52);
	\node (d) at (150:0.6cm){$t$};
	\draw (150:0.48) -- (150:0.52);
	\node (w) at (200:0.6cm){$w$};
	\draw (200:0.48) -- (200:0.52);
	 
	\draw (0:0.5) to[bend left] (110:0.5);
	\draw (0:0.5) to[bend right] (-20:0.5);
	\draw (0:0.5) to[bend right] (-60:0.5);
	\draw (0:0.5) to[bend right] (-40:0.5);
	
	\node (e) at (-65:0.45) {$\cdot$};
	\node (f) at (-71:0.45) {$\cdot$};
	\node (g) at (-77:0.45) {$\cdot$};
	
	\end{tikzpicture}
	\end{center}
\caption{Illustration of the proof of Lemma \ref{L:Ptolemy implies nc^2 = id}.} \label{fig:sutwv}
\end{figure}
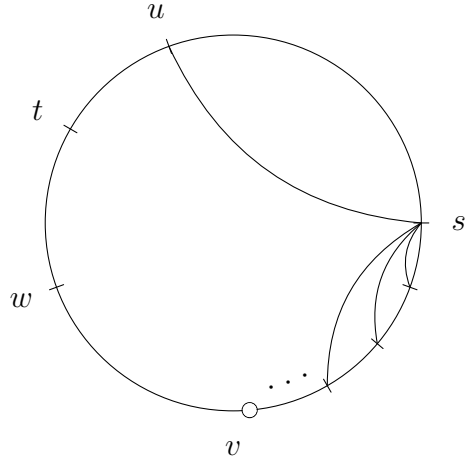
Note that indeed there is a sequence of diagonals $\{s,v_i\}$ from $\X$ with $v_i\to v$ from above, since $v=\inf_{(t,s)}V$ by definition.
	For each $w \in (t,v) \cap \Z$ we have 
		$u < t < w < v < s$,
	so $\{s,t\} \in \nc^2 \X$ crosses $\{u,w\}$ whence $\{u,w\} \notin \nc \X$. So there is 
	a diagonal $\{p,q\} \in \X$ crossing $\{u,w\}$ and we can suppose $p \in [u^+,w^-]$ and $q \in [w^+,u^-]$. 
	
	We claim that $q\not\in (v,u^-]$. Note that $q=s$ is impossible due to the definition of $u$ as 
	supremum and of $v$ as 
	infimum, respectively. Further, if $q\in (v,u^-]$ then $\{p,q\}\in\X$
	crosses $\{s,u\}\in \X$ or one (actually, infinitely many) of $\{s,v_i\}\in \X$. In any
	case, the Ptolemy condition forces $\{s,p\}\in \X$, a contradiction to the choice of $u$
	as supremum or of $v$ as infimum.
	
	So we have shown that $q \in [w^+,v)$. To sum up, we have $t \in (u,v) \cap \Z$ with $u \in \Z$ and $v \in L(\Z)$ and for each $w \in (t,v) \cap \Z$ there exists $\{p,q\} \in \X$
	with
		\[
			u < p < w < q < v.
		\]
	Lemma \ref{L:16} implies that for each $w \in (t,v)\cap \Z$ there exists 
	a diagonal $\{p',q'\} \in \X$ with 
		\begin{equation}\label{Eqn:doublestar}
			u < p' \leqslant t < w < q' < v.
		\end{equation}
	Let $\{w_i\}_{i \in \ZZ_{\geqslant 0}}$ be a sequence in $(t,v)$ with $w_i \to v$ from below. By (\ref{Eqn:doublestar}) we can find a sequence $\{p'_i,q'_i\}_{i \in \ZZ_{\geqslant 0}}$ of diagonals
	from $\X$ with
		\[
			u < p'_i \leqslant t < w_i < q'_i < v,
		\]
	for each $i \in \ZZ_{\geqslant 0}$. 
	It is clear that $q'_i \to v$ from below and by compactness and
	passing to a suitable subsequence we can assume $p'_i \to r$ from below or above for some $r \in [u^+,t]$.
	
	By conditions PC1 and PC2 there is also a sequence $\{p''_i,q''_i\} \in \X$ with $p''_i \to r$ 
	from above
	and $q''_i \to v$ from above. By passing to a subsequence we can assume that $p''_i \in [r,t) \subseteq [u^+,t]$ for each $i \in \mathbb{Z}_{\geqslant 0}$.
	
	But then it is clear that there must exist $j,l\in\mathbb{Z}_{\geqslant 0}$ such that $\{p''_j,q''_j\} \in \X$ crosses $\{s,v_l\}\in \X$ (see Figure \ref{fig:sutwv}).
	Then the Ptolemy condition yields that $\{s,p''_j\}\in \X$, 
	contradicting the definition of $u$ as a supremum.

	Therefore we have finally shown that Subcase B2 cannot occur.
\end{proof}

The following notation will be useful: If $X \subseteq \cT$ is an additive subcategory then we write $\Hom_\T(X,y) = 0$ when $\Hom_\T(x,y) = 0$ for each $x \in X$, and $\Hom_\T(y,X) = 0$ when $\Hom_\T(y,x) = 0$ for each $x \in X$.
We set
\[
  X^\perp = \{y \in \T \mid \Hom_\T(X,y) = 0\}
  \;\;,\;\;
  ^\perp X = \{y \in \T \mid \Hom_\T(y,X) = 0\}. 
\]

The following is Theorem \ref{thm:A} from the introduction.

\begin{Theorem}
\label{T:torsion pair}
Let $\X$ be a set of diagonals of $\cZ$. Then $\add E(\X)$ is the first half of a torsion pair in $\CC(\Z)$ if and only if $\X$ satisfies conditions PC1, PC2, and the Ptolemy condition.
\end{Theorem}

\begin{proof}
	If $Y$ is a diagonal of $\Z$, then by Section \ref{sec:IT}(v) we have $Y \in \nc \X$ if and only if 
		\[
			\Ext_{ \cC( \cZ ) }^1(\add E(\X), E(Y)) = 0,
		\] 
	if and only if $E(Y) \in (\Sigma^{-1} \add E(\X))^\perp$. Symmetrically (recall that $\CC(\Z)$ is 2-Calabi-Yau), $Y \in \nc \X$ if and only if
		\[
			\Ext_{ \cC( \cZ ) }^1(E(Y), \add E(\X)) = 0,
		\]
	if and only if $E(Y) \in {}^\perp(\Sigma \add E(\X))$. Thus, $\X = \nc^2 (\X)$ if and only if
		\[
			\add E(\X) = {}^\perp((\add E(\X))^\perp).
		\]
	Now, by \cite[Proposition~2.3]{IY:mutation}, the subcategory  
	$\add(E(\X))$ is the first half of a torsion pair if and only if $\add(E(\X))$ is precovering and $\add E(\X) = {}^\perp((\add E(\X))^\perp)$,
	which by the above is the case if and only if $\add(E(\X))$ is precovering and $\X = \nc^2 \X$. By Theorem \ref{T:precovering} and Lemmas \ref{L:nc^2 = id implies Ptolemy} and \ref{L:Ptolemy implies nc^2 = id}, this is equivalent to
	$\X$ satisfying conditions PC1, PC2, and the Ptolemy condition.
\end{proof}

\section{Cluster tilting subcategories of the cluster categories $\CC(\Z)$}
\label{sec:cluster_tilting}

This section proves Theorems \ref{thm:B} and \ref{thm:C} from the introduction (=Theorems \ref{T:cluster tilting} and \ref{thm:cluster_structure}).  To set the scene, recall the definition of cluster tilting subcategories of triangulated categories due to Iyama \cite[def.\ 1.1]{I}.

\begin{Definition}
\label{def:cluster_tilting}
Let $\T$ be a triangulated category. A full	subcategory $X \subseteq \T$ is called {\em weakly cluster tilting} if 	$X = (\Sigma^{-1}X)^{\perp} = {}^{\perp}(\Sigma X)$.
	
A subcategory $Y \subseteq \T$ is called {\em cluster tilting} if it is weakly cluster tilting and functorially finite, i.e.\ it is precovering (see Definition \ref{D:precovering}) and preenveloping (for each $t \in T$ there is a morphism $f \colon t \to y$ with $y \in Y$ such that each morphism $t \to y'$ with $y' \in Y$ factors through $f$).
\end{Definition} 
\medskip

\begin{Remark} \label{rem:clustertilting}
By \cite[Lemma~3.2(3)]{KZ:clustertilting} a full subcategory $Y \subseteq \cT$  is cluster tilting if and only if it is weakly cluster tilting and precovering.  So we will not need to consider the preenveloping property.  
\end{Remark}
\medskip

\begin{Lemma}\label{L:upper bounds of arcs ending in an interval}
	Let $\X$ be a set of diagonals of $\Z$ satisfying condition PC1 or condition PC2. 
For $z \in \Z$ and $a \in L(\Z)$, define
		\[
			U = \big\{u \in [z, a) \cap \Z \,\big|\, \{z,u\} \in \X \big\}.
		\]
	Then one of the following happens:
		\begin{enumerate}
		\setlength\itemsep{4pt}
			\item{$\X$ has a fountain at $z$ converging to $a$.}
			\item{$U = \varnothing$.}
			\item{$s = \sup_{[z,a]} U \in \Z$.}
		\end{enumerate}
\end{Lemma}

\begin{proof}
Assume that (ii) and (iii) do not hold. Then there exists a right fountain at
$z$ converging to the supremum $s\in L(\Z)$. By Remark \ref{R:special case half fountain} there is even a fountain at
$z$ converging to $s$. But by definition of $s$ as supremum over the interval $[z,a]$ we must have $s=a$, i.e.\
(i) holds. 
\end{proof}

\begin{Proposition} 
\label{prop:PC2_implies_fountainleapfrog}
Let $\X$ be a maximal set of pairwise non-crossing diagonals of $\cZ$, and suppose that $\X$ satisfies condition PC2.  For each $a \in L(\Z)$, the set $\X$ has a fountain  or a leapfrog converging to $a$.
\end{Proposition}

\begin{proof}
Assume that $\X$ does not have a fountain converging to $a$. We will show that it has a leapfrog converging to $a$.
 
Pick any diagonal $\{x,y\} \in \X$. 
By switching $x$ and $y$ if necessary we can assume $x<y<a$. By assumption, $\X$
does not have a fountain at $x$
converging to $a$. Thus, by Lemma \ref{L:upper bounds of arcs ending in an interval} 
there is a maximal $s_1\in [x,a]\cap \Z$ such that $\{x,s_1\}\in \X$.

We consider the successor $s_1^+\in \Z$ (this exists since $a$ is a limit point, i.e there are infinitely many 
elements of $\Z$ in the interval $[s_1,a)$). The diagonal $\{x,s_1^+\}$ is not in $\X$ (by maximality of $s_1$). On the 
other hand, $\X$ is maximal non-crossing, thus $\{x,s_1^+\}$ must be
crossed by a diagonal from $\X$. However, this diagonal from $\X$ cannot cross $\{x,s_1\}\in \X$ (since $\X$ is non-crossing),
so it must have $s_1$ as one of its endpoints, say $\{s_1,x_1\}\in \X$ crosses $\{x,s_1^+\}$. 

There are now two possibilities, namely $x_1\in (a,x)\cap \Z$ or $x_1\in (s_1^+,a)\cap \Z$.
We claim that, without loss of generality, we can
assume
\begin{equation}
\label{equ:claim}
  x_1\in (a,x).
\end{equation}
Assume to the contrary that $x_1\in (s_1^+,a)\cap \Z$. Then we apply Lemma
\ref{L:upper bounds of arcs ending in an interval} to the interval $[s_1,a]$ (by assumption there
is no fountain at $s_1$ converging to $a$)
and hence 
we can suppose that $x_1$ is maximal in $(s_1^+,a)\cap \Z$ with the property
that $\{s_1,x_1\}\in\X$. Now consider the diagonal $\{x,x_1\}$; it is not in $\X$ (by maximality of $s_1$). 
Since $\X$ is maximal non-crossing, there exists a diagonal in $\X$ crossing $\{x,x_1\}$. But this diagonal is not
allowed to cross $\{x,s_1\}\in \X$ or $\{s_1,x_1\}\in \X$; so this diagonal must have $s_1$ as one of its endpoints. 
Now, by definition of $x_1$ as maximum, the other endpoint of this diagonal is in the interval $(a,x)$. This finishes the
argument for \eqref{equ:claim}.  Thus there is a diagonal $\{s_1,x_1\}\in \X$ with $x_1\in (a,x)$. 

Now we repeat the above argument starting with the diagonal $\{x_1,s_1\}$ instead of $\{x,y\}$. Then we obtain 
a diagonal $\{x_2,s_2\}\in \X$ where $s_2\in (s_1,a)$ and $x_2\in (a,x_1)$.  

Inductively, we obtain two infinite sequences $(s_i)_{i\in \mathbb{Z}_{\geqslant 0}}$ and 
$(x_i)_{i\in \mathbb{Z}_{\geqslant 0}}$ of points in $\Z$ 
such that $x<s_1<s_2<s_3\ldots <a$ and $a<\ldots <x_3<x_2<x_1<x$. Moreover, there exists a corresponding
sequence of diagonals $\{x_i,s_i\}_{i\in \mathbb{Z}_{\geqslant 0}}$ in $\X$. 

The strictly increasing sequence $(s_i)_{i\in \mathbb{Z}_{\geqslant 0}}$ must converge from below
to some limit point $b\in L(\Z)$, and similarly the strictly decreasing sequence $(x_i)_{i\in \mathbb{Z}_{\geqslant 0}}$
must converge from above to some limit point $c\in L(\Z)$.

If $b=a=c$ then the diagonals $\{x_i,s_i\}_{i\in \mathbb{Z}_{\geqslant 0}}$ show that $\X$ has a leapfrog converging to $a$, and we 
are done. 

Otherwise, condition PC2 (which requires two different limit points), 
applied to the diagonals $\{x_i,s_i\}_{i\in \mathbb{Z}_{\geqslant 0}}$, yields a sequence
$\{y_i^0,y_1^0\}$ of diagonals from $\X$ such that $y_i^0\to b$ from above and $y_i^1\to c$ from above. 
But then some diagonals of this sequence obviously cross some of the diagonals $\{x_i,s_i\}$, a contradiction to $\X$ being
non-crossing.  
\end{proof}

The following observation follows easily from the definitions of leapfrog and fountain, see Definition \ref{def:leapfrog_and_fountain}.

\begin{Lemma}
\label{L:leapfrogs and fountains block limit points}
	Let $\X$ be a set of pairwise non-crossing diagonals of $\cZ$ and let $a \in L(\Z)$.
		\begin{enumerate}
		\setlength\itemsep{4pt}
			\item
				Suppose $\X$ has a leapfrog converging to $a$. Then there cannot be a sequence 
				$\{x_i,y_i\}_{i \in \ZZ_{\geqslant 0}}$ 
				of diagonals in $\X$ such that $(x_i)_{i \in \ZZ_{\geqslant 0}}$ converges to $a$ and 
				$(y_i)_{i \in \ZZ_{\geqslant 0}}$ converges to $p$ 
				for some $p \in \overline{\Z}$ with $p \neq a$.
			\item
				Suppose $\X$ has a fountain at $z\in \Z$ converging to $a$. Then there cannot be a sequence 
				$\{x_i,y_i\}_{i \in \ZZ_{\geqslant 0}}$ of diagonals in $\X$ such that $(x_i)_{i \in \ZZ_{\geqslant 0}}$ 
				converges to $a$ and $(y_i)_{i \in \ZZ_{\geqslant 0}}$ converges to $p$ for some 
				$p \in \overline{\Z}$ with $p \neq z$.
		\end{enumerate}
\end{Lemma}

\begin{Proposition}\label{P:noncrossing with isolated limit points implies precovering}
	Let $\X$ be a set of pairwise non-crossing diagonals of $\cZ$. Suppose that for each $a \in L(\Z)$ there is either a fountain 
	or a leapfrog in $\X$ converging to $a$. Then $\X$ satisfies conditions PC1 and PC2. 
\end{Proposition}

\begin{proof}
According to the definition of the conditions PC1 and PC2 (cf. Definition \ref{D:precovering}), 
let $\{x^i_0,x^i_1\}_{i \in \ZZ_{\geqslant 0}}$ be a sequence of diagonals  
from $\X$ with $x^i_0 \to p$ from below and $x^i_1 \to q$ from below or above
	and $p \neq q$.
	
	If $p,q \in \Z$, then $\{x^i_0,x^i_1\}_{i \in \ZZ_{\geqslant 0}}$ is eventually constant 
	and both conditions PC1 and PC2 are 
	trivially satisfied with $x'^i_0 = x^i_0$ and $x'^i_1 = x^i_1$.
	
	If $p \in L(\Z)$ then by Lemma \ref{L:leapfrogs and fountains block limit points}(i), 
	$\X$ cannot have a leapfrog converging to $p$, so by assumption
	$\X$ must have a fountain at some $z \in \Z$ converging to $p$. By 
	Lemma \ref{L:leapfrogs and fountains block limit points}(ii) this forces $q = z$. 
	Therefore $\X$ has a fountain at $z = q$ converging to $p$, so there certainly is a sequence $\{x'^i_0,x'^i_1\}_{i \in \ZZ_{\geqslant 0}}$ from $\X$
	with $x'^i_0 \to p$ and $x'^i_1 \to z = q$ from above: 
	we can even chose $x'^i_1 = z = q$ for each $i \in \ZZ_{\geqslant 0}$.
	
	If $q \in L(\Z)$ then an analogous argument works.
\end{proof}

The following is Theorem \ref{thm:B} from the introduction.

\begin{Theorem}
\label{T:cluster tilting}
Let $\X$ be a set of diagonals of $\Z$. Then $\add E(\X)$ is a cluster tilting sub\-ca\-te\-go\-ry if and only if $\X$ is a maximal set of pairwise non-crossing diagonals, such that for each $a \in L(\Z)$, the set $\X$ has a fountain or a leapfrog converging to $a$.
\end{Theorem}

\begin{proof}
	By Remark \ref{rem:clustertilting}, the subcategory $\add E(\X)$ is cluster tilting if and only if it
	is weakly cluster tilting and precovering. 
	
	It is straightforward from the description of the $\Ext^1$ spaces in Section \ref{sec:IT}(v) that $\add E(\X)$ is weakly cluster tilting if and only if 
	$\X$ is a maximal set of pairwise non-crossing diagonals. 
	
	Recall from Theorem \ref{T:precovering} that $\add E(\X)$ is a precovering subcategory of $\cC( \cZ )$
	if and only if $\X$ satisfies conditions PC1 and PC2.
	
	So it remains to show that if $\X$ is a maximal set of pairwise non-crossing diagonals, then 
	$\X$ satisfies conditions PC1 and PC2 if and only if for each $a \in L(\Z)$, there is a leapfrog
	or a fountain in $\X$ converging to $a$. But these two implications have been shown in
	Propositions \ref{prop:PC2_implies_fountainleapfrog} and 
	\ref{P:noncrossing with isolated limit points implies precovering}, respectively.	
\end{proof}

\begin{Remark}
If $\Z$ has precisely one limit point, then the assertion of Theorem \ref{T:cluster tilting} was already established in \cite[Theorem B]{HJinfinity}.  In fact, the condition of being locally finite appearing there is equivalent to the existence of a leapfrog converging to the unique limit point.
\end{Remark}
\medskip

Figure \ref{fig:not_cluster_tilting} shows an example of a maximal set of pairwise non-crossing diagonals $\X$ of $\cZ$ for which the corresponding subcategory $\add E(\X)$ is not cluster tilting (only weakly cluster tilting).  In fact, neither limit point has a fountain or a leapfrog converging to it.  

Note that $\X$ satisfies condition PC1 (because no sequence of diagonals from $\X$ satisfies the assumption in PC1), but not condition PC2.  This shows that the conclusion of Proposition \ref{prop:PC2_implies_fountainleapfrog} would not be true if only condition PC1 was assumed.

\begin{figure}
\begin{center}
\begin{tikzpicture}[scale=5,cap=round,>=latex]

        \draw[black] (0,0) circle(0.5cm);
        
  \draw[black] (170:0.5) .. controls (170:0.2) and (10:0.2) .. (10:0.5);
  \draw[black] (170:0.5) .. controls (170:0.2) and (20:0.2) .. (20:0.5);
  \draw[black] (20:0.5) .. controls (20:0.2) and (160:0.2) .. (160:0.5);
  \draw[black] (160:0.5) .. controls (160:0.21) and (30:0.21) .. (30:0.5);
  \draw[black] (30:0.5) .. controls (30:0.22) and (150:0.22) .. (150:0.5);
  \draw[black] (150:0.5) .. controls (150:0.23) and (40:0.23) .. (40:0.5);
  \draw[black] (40:0.5) .. controls (40:0.24) and (140:0.24) .. (140:0.5);
  \draw[black] (140:0.5) .. controls (140:0.25) and (50:0.25) .. (50:0.5);
  \draw[black] (50:0.5) .. controls (50:0.26) and (130:0.26) .. (130:0.5);
  \draw[black] (130:0.5) .. controls (130:0.27) and (60:0.27) .. (60:0.5);
  \draw[black] (60:0.5) .. controls (60:0.28) and (120:0.28) .. (120:0.5);
  \draw[black] (120:0.5) .. controls (120:0.3) and (70:0.3) .. (70:0.5);
  \draw[black] (70:0.5) .. controls (70:0.32) and (110:0.32) .. (110:0.5);
  \draw[black] (110:0.5) .. controls (110:0.35) and (80:0.35) .. (80:0.5);
  \draw[black] (80:0.5) .. controls (80:0.37) and (100:0.37) .. (100:0.5);

  \draw[black] (190:0.5) .. controls (190:0.2) and (350:0.2) .. (350:0.5);
  \draw[black] (350:0.5) .. controls (350:0.2) and (200:0.2) .. (200:0.5);
  \draw[black] (200:0.5) .. controls (200:0.2) and (340:0.2) .. (340:0.5);
  \draw[black] (340:0.5) .. controls (340:0.21) and (210:0.21) .. (210:0.5);
  \draw[black] (210:0.5) .. controls (210:0.22) and (330:0.22) .. (330:0.5);
  \draw[black] (330:0.5) .. controls (330:0.23) and (220:0.23) .. (220:0.5);
  \draw[black] (220:0.5) .. controls (220:0.24) and (320:0.24) .. (320:0.5);
  \draw[black] (320:0.5) .. controls (320:0.25) and (230:0.25) .. (230:0.5);
  \draw[black] (230:0.5) .. controls (230:0.26) and (310:0.26) .. (310:0.5);
  \draw[black] (310:0.5) .. controls (310:0.27) and (240:0.27) .. (240:0.5);
  \draw[black] (240:0.5) .. controls (240:0.28) and (300:0.28) .. (300:0.5);
  \draw[black] (300:0.5) .. controls (300:0.3) and (250:0.3) .. (250:0.5);
  \draw[black] (250:0.5) .. controls (250:0.32) and (290:0.32) .. (290:0.5);
  \draw[black] (290:0.5) .. controls (290:0.35) and (260:0.35) .. (260:0.5);
  \draw[black] (260:0.5) .. controls (260:0.37) and (280:0.37) .. (280:0.5);

        \draw[thick,dotted] (7:0.4) -- (353:0.4);
        \draw[thick,dotted] (173:0.4) -- (187:0.4);
        
	\draw (0:0.5) node[fill=white,circle,inner sep=0.065cm] {} circle (0.02cm);	
	\draw (180:0.5) node[fill=white,circle,inner sep=0.065cm] {} circle (0.02cm);	

  \end{tikzpicture}
\caption{A maximal pairwise non-crossing set of diagonals of $\Z$ corresponding to a subcategory of $\CC(\Z)$ which is not cluster tilting.  Neither limit point has a fountain or a leapfrog converging to it.} 
\label{fig:not_cluster_tilting}
\end{center}
\end{figure}
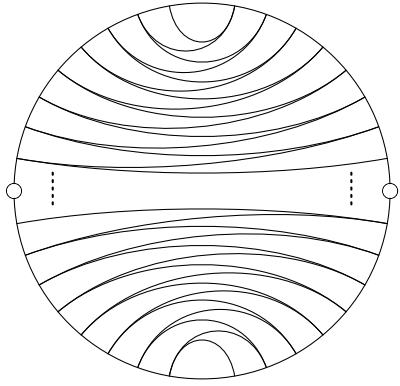

The following is Theorem \ref{thm:C} from the introduction.

\begin{Theorem}
\label{thm:cluster_structure}
The cluster tilting subcategories of $\cC( \cZ )$ form a cluster structure in the sense of \cite[sec.\ II.1]{BIRS}.
\end{Theorem}

\begin{proof}
It is enough to verify the conditions in \cite[thm.\ II.1.6]{BIRS}.

The first condition is that $\cC( \cZ )$ has a cluster tilting subcategory.  This follows from Theorem \ref{T:cluster tilting}.

The second condition is that if $\cT \subseteq \cC( \cZ )$ is a cluster tilting subcategory, then the quiver of $\cT$ has no loops or $2$-cycles.  Recall that up to isomorphism, each indecomposable object of $\cT$ has the form $E(X)$ by Section \ref{sec:IT}(iii).

The space $\Hom_{ \cC( \cZ ) }( E(X),E(X) )$ is $1$-dimensional over the ground field $k$ by Section \ref{sec:IT}(vi), so each non-zero morphism $E(X) \rightarrow E(X)$ is invertible whence the quiver of $\cT$ has no loops.

Let $E( X ) \not\cong E(Y)$ be indecomposable objects in $\cT$ and assume $\Hom_{ \cC( \cZ ) }( E(X),E(Y) ) \neq 0$.  By Section \ref{sec:IT}(vi) we can write $X = \{ x_0,x_1 \}$ and $Y = \{ y_0,y_1 \}$ with
\begin{equation}
\label{equ:x_and_y_inequality}
  x_0 \leqslant y_0 \leqslant x_1^{--} < x_1 \leqslant y_1 \leqslant x_0^{--}.
\end{equation}
If $x_0 \neq y_0$ and $x_1 \neq y_1$ then $X$ and $Y$ would cross, contradicting $\Ext_{ \cC( \cZ ) }^1( E(X),E(Y) ) = 0$ which holds since $E(X),E(Y) \in \cT$.  Without loss of generality we can suppose
\begin{equation}
\label{equ:x_and_y_equality}
  x_0 = y_0
  \;\;\;\;\mbox{and}\;\;\;\;
  x_1 \neq y_1.
\end{equation}  
Suppose we had $\Hom_{ \cC( \cZ ) }( E(Y),E(X) ) \neq 0$.  By Section \ref{sec:IT}(vi) again we would have
\[
  y_0 \leqslant x_0 \leqslant y_1^{--} < y_1 \leqslant x_1 \leqslant y_0^{--}
  \;\;\;\; \mbox{or} \;\;\;\;
  y_0 \leqslant x_1 \leqslant y_1^{--} < y_1 \leqslant x_0 \leqslant y_0^{--}.
\]
But each is incompatible with the combination of \eqref{equ:x_and_y_inequality} and \eqref{equ:x_and_y_equality}, so $\Hom_{ \cC( \cZ ) }( E(Y),E(X) ) = 0$.  Hence there is no $2$-cycle between $E(X)$ and $E(Y)$ in the quiver of $\cT$.
\end{proof}

\begin{Remark}\label{R:clusters}
For most admissible sets $\cZ$, the cluster structure in Theorem \ref{thm:cluster_structure} is different from the one in \cite[Theorem~2.4.1]{IT:cyclicposets}, where the clusters are not necessarily cluster tilting subcategories.

Namely, the convergence condition in \cite[Theorem~2.4.1]{IT:cyclicposets} only asks that for each right (respectively left) fountain at a point $z \in \Z$ converging to a limit point $a \in L(\Z)$, there be a left (respectively right) fountain at $z$ converging to the same limit point $a$ (cf.\ \cite[Definition~2.4.6]{IT:cyclicposets}).

In fact, the clusters in \cite[Theorem~2.4.1]{IT:cyclicposets} coincide with cluster tilting subcategories if and only if $\Z$ is finite or has exactly one limit point. Figure \ref{fig:not_cluster_tilting} yields an example of a cluster in the sense of \cite[Theorem~2.4.1]{IT:cyclicposets} (there is no right or left fountain, so the condition in \cite[Definition~2.4.6]{IT:cyclicposets} is empty) which does not correspond to a cluster tilting subcategory. 
\end{Remark}
\medskip

\medskip
\noindent
{\bf Acknowledgements.}
We thank Charles Paquette, Adam-Christiaan van Roosmalen, and Bin Zhu for illuminating comments on a preliminary version, and the referee for a careful reading and several useful suggestions which have improved the presentation.

This project was supported by grant HO 1880/5-1 under the research priority programme SPP 1388 ``Darstellungstheorie'' of the DFG, and by grant EP/P016014/1 ``Higher Dimensional Homological Algebra'' from the EPSRC.



\begin{thebibliography}{19}

\bibitem{BZ}  T.\ Br\"{u}stle and J.\ Zhang, {\it On the cluster category of a marked surface without punctures}, Algebra and Number Theory {\bf 5} (2011), 529--566.

\bibitem{BIRS} A.\ B.\ Buan, O.\ Iyama, I.\ Reiten, and J.\ Scott, {\it Cluster structures for 2-Calabi--Yau categories and unipotent groups}, Compositio Math.\ {\bf 145} (2009), 1035--1079.

\bibitem{BMRRT}  A.\ B.\ Buan, R.\ J.\ Marsh, M.\ Reineke, I.\ Reiten,
and G.\ Todorov, {\it Tilting theory and cluster combinatorics}, Adv.\
Math.\ {\bf 204} (2006), 572--618. 

\bibitem{CCS}  P.\ Caldero, F.\ Chapoton, and R.\ Schiffler, {\it Quivers with relations arising from clusters ($A_n$ case)}, Trans.\ Amer.\ Math.\ Soc.\ {\bf 358} (2006), 1347--1364.

\bibitem{CZZ}  H.\ Chang, Y.\ Zhou, and B.\ Zhu, {\it Cotorsion pairs in cluster categories of type $A^{ \infty }_{ \infty }$}, J.\ Combin.\ Theory Ser.\ A {\bf 156} (2018), 119--141.

\bibitem{Dickson:torsiontheory}
 	S.\ E.\ Dickson,
 	{\it A torsion theory for abelian categories,}
 	Trans.\ Amer.\ Math.\ Soc.\ {\bf 121} (1966), 223--235.
 
\bibitem{E}  E.\ E.\ Enochs, {\it Injective and flat covers, envelopes
    and resolvents}, Israel J.\ Math.\ {\bf 39} (1981), 189--209. 
 
 \bibitem{Happel:triangcats}
 	 D.\ Happel,
 	``Triangulated categories in the representation theory of finite dimensional algebras'',
 	London Math.\ Soc.\ Lecture Note Ser., Vol. 119, Cambridge University Press, Cambridge, 1988.
 	 
 \bibitem{HJinfinity} 
 	T.\ Holm and P.\ J\o rgensen,
 	{\it On a cluster category of infinite Dynkin type, and the relation to triangulations of the infinity-gon,} 
 	 Math.\ Z.\ {\bf 270} (2012), 277--295. 
 	 
 \bibitem{HJR:torsionpairsA}
 	T.\ Holm, P.\ J\o rgensen, and M.\ Rubey,
 	{\it Ptolemy diagrams and torsion pairs in the cluster category of Dynkin type $A_n$,} 
 	J.\ Algebraic Combin.\ {\bf 34} (2011), 507--523.

 \bibitem{HJR:tubes}
 	T.\ Holm, P.\ J\o rgensen, and M.\ Rubey,
    {\it Torsion pairs in cluster tubes}, 
 	J.\ Algebraic Combin.\ {\bf 39} (2014), 587--605.
 	 
 \bibitem{IT:cyclicposets}
 	K.\ Igusa and G.\ Todorov,
 	{\it Cluster categories coming from cyclic posets,}
 	 Comm.\ Algebra {\bf 43} (2015), 4367--4402.

\bibitem{I}  O.\ Iyama, {\it Maximal orthogonal subcategories of triangulated categories satisfying Serre duality,} in: Oberwolfach Reports, Vol.\ 6, European Mathematical Society, 2005.

 \bibitem{IY:mutation} 
 	O.\ Iyama and Y.\ Yoshino,
 	{\it Mutation in triangulated categories and rigid Cohen-Macaulay modules,}
 	Invent.\ Math.\ {\bf 172} (2008), 117--168.
 	
 \bibitem{KZ:clustertilting}
 	S.\ Koenig and B.\ Zhu,
 	{\it From triangulated categories to abelian categories--cluster tilting in a general framework,}
 	Math.\ Z.\ {\bf 258} (2008), 143--160.

\bibitem{LP}  S.\ Liu and C.\ Paquette, {\it Cluster categories of type $A_{ \infty }^{ \infty }$ and triangulations of the infinite strip}, Math.\ Z.\ {\bf 286} (2017), 197--222.

\bibitem{Ng}  P.\ Ng, A characterization of torsion theories in the cluster category of Dynkin type $A_{ \infty }$, preprint (2010).  {\tt arXiv:1005.4364v1.}

\bibitem{QZ}  Y.\ Qiu and Y.\ Zhou, {\it Cluster categories for marked surfaces: punctured case}, Compositio Math.\ {\bf 153} (2017), 1779--1819.

\bibitem{SvR}  J.\ \v{S}\v{t}ov\'{\i}\v{c}ek and A.-C.\ van Roosmalen, $2$-Calabi--Yau categories with a directed cluster-tilting subcategory, preprint (2016).  {\tt arXiv:1611.03836v1.}
 	
\bibitem{ZZZ}  J.\ Zhang, Y.\ Zhou, and B.\ Zhu, {\it Cotorsion pairs in the cluster category of a marked surface}, J.\ Algebra {\bf 391} (2013), 209--226.

 	
 \end{thebibliography}
\end{document}